\documentclass[a4paper,12pt]{amsart}

\usepackage{amsmath}
\usepackage{amssymb}
\usepackage{amsthm}
\usepackage[a4paper]{geometry}
\usepackage{todonotes}
\usepackage{enumerate}   
\usepackage{cite}
\usepackage{mathrsfs}

\makeatletter
\@namedef{subjclassname@2020}{%
  \textup{2020} Mathematics Subject Classification}
\makeatother

\theoremstyle{plain}
\newtheorem{thm}{Theorem}[section]
\newtheorem{cor}{Corollary}[thm]
\newtheorem{lem}[thm]{Lemma}
\newtheorem{prop}[thm]{Proposition}

\theoremstyle{definition}

\newtheorem{exmp}{Example}[section]

\theoremstyle{remark}
\newtheorem{rem}{Remark}

\newcommand\inner[2]{\left\langle #1, #2 \right\rangle}

\newcommand{\tco}{\mathcal{S}}
\newcommand{\nuc}{\mathcal{N}(L^2,M^1_v)}
\newcommand{\bo}{\mathcal{L}(L^2)}

\newcommand{\beauty}{\mathcal{B}}

\newcommand{\R}{\mathbb{R}}
\newcommand{\Rd}{\mathbb{R}^d}
\newcommand{\Rdd}{\mathbb{R}^{2d}}
\newcommand{\Z}{\mathbb{Z}}
\newcommand{\N}{\mathbb{N}}

\newcommand{\HS}{\mathcal{HS}}
\newcommand{\F}{\mathcal{F}}
\newcommand{\tr}{\mathrm{tr}}
\newcommand{\kernel}{k}
\newcommand{\weyl}{a}
\newcommand{\opstft}{\mathfrak{V}}

\title{Equivalent norms for modulation spaces from positive Cohen's class distributions}
\author{Eirik Skrettingland}
\address{Department of Mathematics\\ NTNU Norwegian University of Science and
Technology\\ NO–7491 Trondheim\\Norway}
\email{eirik.skrettingland@ntnu.no}
\begin{document}

\begin{abstract}
	We give a new class of equivalent norms for modulation spaces by replacing the window of the short-time Fourier transform by a Hilbert-Schmidt operator. The main result is applied to Cohen's class of time-frequency distributions, Weyl operators and localization operators. In particular, any positive Cohen's class distribution with Schwartz kernel can be used to give an equivalent norm for modulation spaces.   We also obtain a description of modulation spaces as time-frequency Wiener amalgam spaces. The Hilbert-Schmidt operator must satisfy a nuclearity condition for these results to hold, and we investigate this condition in detail. 
	\end{abstract}
\maketitle \pagestyle{myheadings} \markboth{Eirik Skrettingland}{Equivalent norms for modulation spaces from Cohen's class}
\thispagestyle{empty}
\section{Introduction}

The modulation spaces introduced by Hans Feichtinger \cite{Feichtinger:1983} have long been recognized as suitable function spaces for various problems in time-frequency analysis\cite{Grochenig:2001,Feichtinger:1997}, PDEs\cite{Benyi:2007,Wang:2007}, pseudodifferential operators \cite{Benyi:2005,Cordero:2003,Grochenig:2006pd,Toft:2004} and others areas -- comprehensive lists of references can be found in \cite{Feichtinger:2006lba} and the recent monograph \cite{Benyi:2020}. Perhaps the most common definition of the modulation spaces nowadays uses the language of \textit{time-frequency analysis}. To motivate the definition, we consider a function $\psi$ on $\Rd$ and its Fourier transform $$\hat{\psi}(\omega)=\int_{\Rd} \psi(t) e^{-2\pi i \omega \cdot t} \ dt \quad \text{ for } \omega \in \Rd.$$ Together, $\psi$ and $\hat{\psi}$ describe the behaviour of $\psi$ as a function of time and frequency, respectively, and give us different approaches to study properties of $\psi$. For instance, smoothness of $\psi$ is related to decay of $\hat{\psi}$.  But although $\hat{\psi}$ shows which frequencies $\omega$ contribute to $\psi$ -- those such that $|\hat{\psi}(\omega)|$ is large -- it does not indicate \textit{when}, i.e. for which $t\in \Rd$, the frequency contributes to $\psi$. In time-frequency analysis one therefore looks for \textit{time-frequency distributions} $Q(\psi)$, which should be a function on $\Rdd$ such that the size of $Q(\psi)(x,\omega)$ describes the contribution of frequency $\omega$ at time $x$ in $\psi$.

The existence of an ideal time-frequency distribution $Q$ is prohibited by various uncertainty principles, but a common choice in time-frequency analysis is the \textit{short-time Fourier transform} (STFT)
$$V_\varphi \psi(z)=\inner{\psi}{\pi(z)\varphi}_{L^2} \text{ for } z\in \Rdd,$$ where the \textit{window} $\varphi$ is a function on $\Rd$ well-localized in time and frequency, and $\pi(z)$ denotes the \textit{time-frequency shift} $$\pi(z)\varphi(t)=e^{2\pi i \omega \cdot t}\varphi(t-x).$$

The modulation spaces $M^{p,q}_m(\Rd)$ are then defined, for $1\leq p,q\leq \infty$ and a weight function $m$ on $\Rdd$, by the norm
\begin{equation} \label{eq:intromod}
  \|\psi\|_{M^{p,q}_m}=\left(\int_{\Rd} \left(\int_{\Rd} |V_{\varphi_0}\psi(x,\omega)|^p m(x,\omega)^p \ dx \right)^{q/p} \ d\omega \right)^{\frac{1}{q}},
\end{equation}
where $\varphi_0(t)=2^{d/4}e^{-\pi |t|^2}$ and the integrals are replaced by supremums for $p,q=\infty$. By our interpretation of $V_{\varphi_0} \psi(x,\omega)$ as a time-frequency distribution, we see that $\|\psi\|_{M^{p,q}_m}$ measures how localized $\psi$ is in the time-frequency plane. More precisely, $L^p$ measures the decay of $\psi$ in time, and $L^q$ the decay of $\psi$ in frequency -- i.e. the decay of $\hat{\psi}$, or the smoothness of $\psi$. The fact that $\|\psi\|_{M^{p,q}_m}$ is finite is therefore a statement on the decay and smoothness of $\psi$.

A useful result on modulation spaces from \cite{Feichtinger:1983} is that replacing the window $\varphi_0$ in \eqref{eq:intromod} by another window $\varphi$ with good time-frequency localization, we obtain an equivalent norm on $M^{p,q}_m(\Rd)$:
\begin{equation}  \label{eq:intronewwindow}
  \|\psi\|_{M^{p,q}_m}\asymp \left(\int_{\Rd} \left(\int_{\Rd} |V_{\varphi}\psi(x,\omega)|^p m(x,\omega)^p \ dx \right)^{q/p} \ d\omega \right)^{\frac{1}{q}}.
\end{equation} The main result of this contribution is an extension of this fact: we show that the window can even be replaced by a Hilbert-Schmidt operator $S$ on $L^2(\Rd)$. 
To explain this transition from function-windows to operator-windows, we fix an arbitrary $\xi \in L^2(\Rd)$ with $\|\xi\|_{L^2}=1$ and  consider the rank-one operator $S=\xi \otimes \varphi$ defined by
\begin{equation}\label{eq:introrankone}
S(\psi)=\xi \otimes \varphi(\psi)=\inner{\psi}{\varphi}_{L^2} \xi.
\end{equation}
It is easy to see that $\|S\pi(z)^*\psi\|_{L^2}=|V_\varphi \psi(z)|,$ hence we may reformulate \eqref{eq:intronewwindow} as 
\begin{equation} \label{eq:intromain}
    \|\psi\|_{M^{p,q}_m}\asymp  \left(\int_{\Rd} \left(\int_{\Rd} \|S\pi(z)^*\psi\|_{L^2}^p m(x,\omega)^p \ dx \right)^{q/p} \ d\omega \right)^{\frac{1}{q}}. 
\end{equation}
Our main result in Theorem \ref{thm:continuous} states that this holds not only for rank-one $S$ as in \eqref{eq:introrankone}, but for all Hilbert-Schmidt operators $S$ having good time-frequency localization -- a statement that itself will need elaboration. By choosing different $S$ we will see that we obtain equivalent norms for the modulation spaces that express quite different properties from these expressed in \eqref{eq:intromod}, hence giving new insights into the structure of modulation spaces.

Comparing \eqref{eq:intromod} and \eqref{eq:intromain}, we see that $|V_{\varphi}\psi(z)|$ is replaced by $\|S\pi(z)^*\psi\|_{L^2}$. This suggests that we replace the STFT by the function $\opstft_S:\Rdd\to L^2(\Rd)$ given by $$\opstft_S(\psi)(z)=S\pi(z)^*\psi.$$
In Section \ref{sec:tfa} we show that $\opstft_S$ actually behaves like the usual STFT $V_\varphi$, by showing that it satisfies an isometry property and an inversion formula. This insight allows us to prove \eqref{eq:intromain} in Section \ref{sec:mainresult} using methods similar to those used to prove that the modulation spaces are independent of the window function in \cite{Grochenig:2001}. 

 Sections \ref{sec:bonychemin}, \ref{sec:cohensclass} and \ref{sec:locops} are then devoted to examples and reinterpretations of the main result. First we consider Weyl operators in Section \ref{sec:bonychemin}. The reformulation of \eqref{eq:intromain} in Theorem \ref{thm:bonychemin} generalizes a result by Gr\"ochenig and Toft \cite{Grochenig:2011toft} that identifies certain modulation spaces with function spaces introduced by Bony and Chemin\cite{Bony:1994}.
 
 In Section 7 we turn our attention to Cohen's class of time-frequency distributions. As there is no ideal time-frequency disribution, Cohen's class was introduced by Cohen in \cite{Cohen:1966} as the time-frequency distributions $Q_a$ given by 
 $$Q_a(\psi)(z)=a\ast W(\psi) \quad \text{ for } z\in \Rdd,$$
 where $a$ is some function (or distribution) on $\Rdd$ and $W(\psi)$ is the Wigner-distribution, see \eqref{eq:wigner} for its definition. By varying $a$ one obtains time-frequency distributions with different properties. An important example of a Cohen's class distribution is the \textit{spectrogram} $Q(\psi)(z)=|V_{\varphi_0}\psi(z)|^2$. Then \eqref{eq:intromod} shows that the modulation space norm of $\psi$ is given by the $L^{p,q}_m$-norm of (the square root of) $Q(\psi)$. We might therefore ask whether this is true if we replace the spectrogram by another Cohen's class distributions $Q_a$. Using a description of Cohen's class in terms of bounded operators given in \cite{Luef:2018b} together with \eqref{eq:intromain}, we are able to give in Theorem \ref{thm:cohenclass} a set of Cohen's class distributions whose $L^{p,q}_m$ norms define the modulation space norms. The question of characterizing these Cohen class distributions $Q_a$ in terms of $a$ seems to be a difficult problem in general. However, using a result from \cite{Keyl:2015} we are able to prove the following in Theorem \ref{thm:squareroot}:
 \begin{quote}
 	Let $1\leq p,q \leq \infty$ and assume that the weight $m$ grows at most polynomially. If $a$ is a Schwartz function on $\Rdd$ and $Q_a(\psi)$ is a positive function for each $\psi \in L^2(\Rd)$, then
 	\begin{equation*}
  \|\psi\|_{M^{p,q}_m} \asymp \left\|\sqrt{Q_{a}(\psi)}\right\|_{L^{p,q}_{m}(\Rdd)} .
\end{equation*}
 \end{quote}
 Finally, we let $S$ in \eqref{eq:intromain} be a \textit{localization operator} in Section 8. This leads to a characterization of modulation spaces as time-frequency Wiener amalgam spaces in Theorem \ref{thm:locop}, which is a continuous version of results by D\"orfler, Feichtinger and Gr\"ochenig \cite{Dorfler:2011,Dorfler:2006}, see also \cite{Romero:2012,Dorfler:2014}, much like the fact that the standard Wiener amalgam spaces have both a continuous and discrete description. We mention that \cite{Boggiatto:2005,Grochenig:2011toft,Grochenig:2013} also use localization operators to get equivalent norms for modulation spaces, but their approach and results are different from those we consider. 

Before ending this introduction, we wish to point out that sufficient conditions on $S$ for \eqref{eq:intromain} to hold will be a recurring theme throughout the paper. The most general sufficient condition on $S$ is that its Hilbert space adjoint must be a nuclear operator from $L^2(\Rd)$ to $M^1_v(\Rd)$. In some ways this a very natural condition: if applied to the rank-one operator in \eqref{eq:introrankone} it means that $\varphi \in M^1_v(\Rd)$, which is the standard condition for windows for modulation spaces. As we see in Section \ref{sec:operators}, this nuclearity condition is also easy to handle when working with localization operators. From other perspectives, such as the Weyl calculus, the condition is more mysterious, and we will therefore also study stronger sufficient conditions on $S$ for \eqref{eq:intromain} to hold.

\subsection*{Notation and conventions}

If $X$ is a Banach space and $X'$ its dual space, the action of $y\in X'$ on $x\in X$ is denoted by the bracket $\inner{y}{x}_{X',X}$, where the bracket is antilinear in the second coordinate to be compatible with the notation for inner products in Hilbert spaces. This means that we are identifying the dual space $X'$ with \textit{anti}linear functionals on $X$. For two Banach spaces $X,Y$ we denote by $\mathcal{L}(X,Y)$ the Banach space of bounded linear operators $S:X\to Y$, and if $X=Y$ we simply write $\mathcal{L}(X)$. For brevity we often write $\bo$ for $\mathcal{L}(L^2(\Rd)$. For topological spaces $X,Y$ we write $X\hookrightarrow Y$ to denote that there is a continuous inclusion of $X$ into $Y.$

For $p\in [1,\infty]$, $p'$ denotes the conjugate exponent, i.e. $\frac{1}{p}+\frac{1}{p'}=1$. The notation $P \lesssim Q$ means that there is some $C>0$ such that $P\leq C\cdot Q$, and $P\asymp Q$ means that $Q\lesssim P$ and $P\lesssim Q$. For $\Omega \subset \Rdd$, $\chi_\Omega$ is the characteristic function of $\Omega.$ $\mathscr{S}(\Rd)$ denotes the Schwartz space, and $\mathscr{S}'(\Rd)$ its dual space  of tempered distributions.
\section{Time-frequency analysis}

As we have seen in the introduction, our main results are phrased in terms of the \textit{time-frequency shifts} $\pi(z)\in \bo$ for $z=(x,\omega)\in \Rdd$, defined by $$\pi(z)\psi(t)=e^{2\pi i \omega \cdot t} \psi(t-x) \quad \text{ for } \psi \in L^2(\Rd).$$ The time-frequency shifts are unitary on $L^2(\Rd)$, and for they satisfy 
\begin{align}
\pi(x,\omega)\pi(x',\omega')&=e^{-2\pi i \omega'\cdot x}\pi(x+x',\omega+\omega') \label{eq:tfcomposition}\\
\pi(x,\omega)^* &= e^{-2 \pi i x \cdot \omega} \pi(-x,-\omega) \label{eq:tfadjoint}
\end{align}
for $x,x',\omega, \omega' \in \Rd$. Closely related to the time-frequency shifts is the short-time Fourier transform (STFT) $V_\varphi \psi\in L^2(\Rdd)$, given by 
\begin{equation} \label{eq:STFT}
  V_\varphi \psi(z) =\inner{\psi}{\pi(z) \varphi}_{L^2} \quad \text{ for } \psi,\phi\in L^2(\Rd), z\in \Rdd. 
\end{equation}
The function $\varphi$ is often referred to as the \textit{window} of the STFT $V_\varphi \psi$. An important property of the STFT is \textit{Moyal's identity}\cite[Thm. 3.2.1]{Grochenig:2001}.
\begin{lem}[Moyal's identity]
If $\psi_1, \psi_2, \phi_1, \phi_2 \in L^2(\R^d)$, then $V_{\phi_i}\psi_j \in L^2(\R^{2d})$ for $i,j\in \{1,2\}$ and 
\begin{equation*}
	\int_{\Rdd} V_{\phi_1} \psi(z) \overline{V_{\phi_2}\psi(z)} \ dz=\inner{\psi_1}{\psi_2}_{L^2}\overline{\inner{\phi_1}{\phi_2}}_{L^2}.
\end{equation*}
\end{lem}
In particular, we see that for fixed window $\varphi$ with $\|\varphi\|_2=1$ the map $\psi \mapsto V_\varphi \psi$ is an isometry from $L^2(\Rd)$ to $L^2(\Rdd)$.

\subsection{Admissable weight functions and weighted, mixed $L^{p}$ spaces}
A \textit{submultiplicative} weight function $v$ on $\Rdd$ is a non-negative function $v:\Rdd \to \R$ such that $v(z_1+z_2)\leq v(z_1)v(z_2)$ for $z_1,z_2 \in \Rdd$. Whenever we refer to a submultiplicative weight function $v$ we will assume that $v$ is continuous and satisfies $v(x,\omega)=v(-x,\omega)=v(x,-\omega)=v(-x,-\omega)$; these assumptions do not lead to a loss of generality as any submultiplicative weight function is equivalent in a natural sense to a weight satisfying these assumptions, see \cite{Grochenig:2001,Grochenig:2007}. Furthermore, these assumptions imply that if  $v$ is not identically $0$, then $v(z)\geq 1$ for all $z\in \Rdd$.
 The assumptions above are satisfied by standard examples such as the polynomial weights $$v_s(z)=(1+|z|^2)^{s/2} \quad s\geq 0,$$ but also by the exponential weights $v_a(z)=e^{a|z|}$ for $a\geq 0$. 
 A non-negative weight function $m$ on $\Rdd$ is said to be $v$-\textit{moderate} if $v$ is a submultiplicative weight function and there exists some constant $C_v^m>0$ such that 
\begin{equation*} 
  m(z_1+z_2)\leq C_v^m v(z_1)m(z_2).
\end{equation*}
 We refer the reader to the survey \cite{Grochenig:2007} for more examples and motivation for these assumptions. For any $v$-moderate weight $m$ and $1\leq p,q\leq \infty$ we may define the Banach space $L^{p,q}_m(\Rdd)$ to be the equivalence classes of Lebesgue measurable functions $F:\Rdd\to \mathbb{C}$ such that
\begin{equation*}
  \|F\|_{L^{p,q}_{m}} :=\left(\int_{\Rd} \left(\int_{\Rd} |F(x,\omega)|^p m(x,\omega)^p \ dx \right)^{q/p} \ d\omega \right)^{\frac{1}{q}} < \infty.
\end{equation*}
If $p<\infty$ or $q<\infty$, the corresponding integral is replaced by an essential supremum.

\subsection{Modulation spaces}
 Throughout the rest of the paper, we will let $\varphi_0\in L^2(\Rd)$ denote the normalized Gaussian, i.e. $$\varphi_0(t)=2^{d/4} e^{-\pi t\cdot t}\quad \text{ for }t\in \Rd.$$ For a submultiplicative weight $v$, we define the space $M^1_v(\Rd)$ to be the Banach space of those $\psi\in L^2(\Rd)$ such that 
\begin{equation*}
  \|\psi\|_{M^1_v} :=\int_{\Rdd} |V_{\varphi_0}\psi(z)| v(z) \ dz<\infty.
\end{equation*}
This will serve as our space of test functions. It is always non-empty as it contains $\varphi_0$, and for weights $v$ of polynomial growth it contains the Schwartz functions $\mathscr{S}(\Rd)$\cite[Prop. 11.3.4]{Grochenig:2001}. For more general weights $M^1_v(\Rd)$ will not necessarily contain $\mathscr{S}(\Rd)$ and might be quite small. The time-frequency shifts $\pi(z)$ are bounded on $M^1_v(\Rd)$ \cite[Thm. 11.3.5]{Grochenig:2001} with 
\begin{equation} \label{eq:tfbound}
  \|\pi(z)\psi\|_{M^1_v}\leq v(z) \|\psi\|_{M^1_v},
\end{equation}
and hence the STFT $V_\phi \psi(z)$ for $\phi\in M^1_v(\Rd)$ and $\psi \in (M^1_v(\Rd))'$ can be defined by modifying the inner product in the definition \eqref{eq:STFT} to a duality bracket: $V_\phi\psi(z)=\inner{\psi}{\pi(z)\phi}_{(M^1_v)',M^1_v}$. 

For any $v$-moderate weight $m$ and $1\leq p,q \leq \infty$, we then define the modulation space $M^{p,q}_m(\Rd)$ to consist of those $\psi \in (M^1_v(\Rd))'$ such that 
\begin{equation*}
  \|\psi\|_{M^{p,q}_m}:=\|V_{\varphi_0}\psi\|_{L^{p,q}_m} <\infty.
\end{equation*}
 When $p=q$ we will write $M^p_m(\Rd)$ for $M^{p,p}_m(\Rd)$, and when $m\equiv 1$ we write $M^{p,q}(\Rd)$. Some properties of the modulation spaces are summarized below, proofs may be found in the monograph \cite{Grochenig:2001}.

\begin{prop} \label{prop:modulationspaces}
Let $m$ be a $v$-moderate weight and $1\leq p,q \leq \infty$.
	\begin{enumerate}[(a)]
		\item $M^{p,q}_{m}(\Rd)$ is a Banach space with the norm $\|\cdot \|_{M^{p,q}_m}$.
		\item If $1\leq p_1\leq p_2\leq \infty$, $1\leq q_1\leq q_2\leq \infty$ and $m_2\lesssim m_1$, then $M_{m_1}^{p_1,q_1}(\Rd)\hookrightarrow M_{m_2}^{p_2,q_2}(\Rd)$. 
		\item If $p,q<\infty$, then $M_{1/m}^{p',q'}(\Rd)$ is the dual space of $M^{p,q}_{m}(\Rd)$ with
 \begin{equation*} 
  \inner{\phi}{\psi}_{M^{p',q'}_{1/m},M^{p,q}_m}=\int_{\Rdd} V_{\varphi_0}\psi(z)\overline{V_{\varphi_0}\phi(z)} \ dz.
\end{equation*}
		\item $L^2(\Rd)=M^2(\Rd)$.
	\end{enumerate}
\end{prop}

\begin{rem} 
\begin{enumerate}[(a)]
\item As a particular case of part $c)$, we may identify $(M^1_v(\Rd))'$ with $M^\infty_{1/v}(\Rd)$, which we will do for the rest of the paper. The reader should also note that the duality extends the inner product on $L^2(\Rd)$, since if $\psi \in L^2(\Rd)\cap M^\infty_{1/v}(\Rd)$ and $\phi\in M^1_v(\Rd)$, we find by Moyal's identity that 
\begin{equation*}
  \inner{\psi}{\phi}_{M^\infty_{1/v},M^1_v}=\int_{\Rdd} V_{\varphi_0}\phi(z)\overline{V_{\varphi_0}\psi(z)} \ dz=\inner{\psi}{\phi}_{L^2}.
\end{equation*}
    \item A simple calculation using our assumption that $v(-z)=v(z)$ gives that if $m$ is $v$-moderate, then so is $1/m$. 
	\item As mentioned, $\mathscr{S}(\Rd)$ embeds continuously into $M^1_v(\Rd)$ when $v$ grows polynomially, so in this case we may identify $M^\infty_{1/v}(\Rd)$ with a subspace of the tempered distributions. This is not true for more general weights, hence we need to work with the abstract space $M^\infty_{1/v}(\Rd)$ defined as the dual space of our test functions $M^1_v(\Rd)$.
\end{enumerate}
\end{rem}

 The property of modulation spaces that is our main focus is the fact that changing the window for the STFT leads to an equivalent norm.

\begin{thm} \label{thm:windowindependent}
	Let $m$ be a $v$-moderate weight function and let $\phi\in M^1_v(\Rd)$. Then $\|V_\phi\psi\|_{L^{p,q}_m}$ defines an equivalent norm on $M^{p,q}_m(\Rd)$: for $\psi \in M^{p,q}_m(\Rd)$ we have $$\|V_\phi \psi\|_{L^{p,q}_m}\asymp \|\psi\|_{M^{p,q}_m}.$$
\end{thm}

Our main result is that we also obtain equivalent norms for $M^{p,q}_m(\Rd)$ when $\phi$ is replaced by an operator $S$ satisfying certain conditions, after modifying the definition of the STFT correspondingly. To prove this, we will use the precise statement of the upper bound $\|V_{\phi}\psi\|_{L^{p,q}_m}\lesssim \|\psi\|_{M^{p,q}_m}$; it follows from equation (11.33) in \cite{Grochenig:2001}. Recall that $C_v^m$ is the constant from $m(z_1+z_2)\leq C_v^m v(z_1)m(z_2).$

\begin{prop} \label{prop:rankonecase}
Let $m$ be a $v$-moderate weight function and let $\phi\in M^1_v(\Rd)$. The map $\psi \mapsto V_\phi \psi$ is bounded from $M^{p,q}_m(\Rd)$ to $L^{p,q}_m(\Rdd)$ with $ \|V_\phi \psi\|_{L^{p,q}_m}\leq C_v^m \|\phi\|_{M^1_v} \|\psi\|_{M^{p,q}_m}$.
\end{prop}

\section{Classes of operators for time-frequency analysis} \label{sec:operators}
Our main result rests upon properties of certain classes of operators, all of which may be described as integral operators.
\subsection{Hilbert-Schmidt operators} \label{sec:hilbertschmidt}

 Given a function $k\in L^2(\Rdd)$, we define the bounded integral operator $T_k:L^2(\Rd)\to L^2(\Rd)$ by
\begin{equation*}
  T_k(\psi)(x)=\int_{\Rd} k(x,y)\psi(y) \ dy \quad \text{ for } \psi \in L^2(\Rd).
\end{equation*}
We call $k$ the integral kernel of the operator $T_k$. When equipped with the inner product 
\begin{equation*}
  \inner{T_{k_1}}{T_{k_2}}_{\HS}:= \inner{k_1}{k_2}_{L^2},
\end{equation*}
 the set of integral operators $T_k$ with integral kernels $k\in L^2(\Rdd)$ forms a Hilbert space of compact operators called the \textit{Hilbert-Schmidt operators}, which we will denote by $\HS$. Given $T\in \HS$, we will sometimes denote its integral kernel by $\kernel_T$, which means that $T=T_{\kernel_T}$. An important subspace of $\HS$ is the space of \textit{trace class operators}, consisting of those $T\in \HS$ such that 
\begin{equation*}
  \sum_{n=1}^\infty \inner{|T|e_n}{e_n}_{L^2}<\infty,
\end{equation*}
where $\{e_n\}_{n=1}^\infty$ is any orthonormal basis of $L^2(\Rd)$ and $|T|$ is the positive part in the polar decomposition of $T$. If $T$ is a trace class operator, we may therefore define its \textit{trace} $\tr(T)$ by 
\begin{equation*}
  \tr(T)=\sum_{n=1}^\infty \inner{Te_n}{e_n}_{L^2},
\end{equation*}
 which can be shown to be independent of the orthonormal basis. For our part, we will need that if $S,T\in \HS$, then $ST$ is a trace class operator. In particular, this allows us to express the inner product on $\HS$ without reference to their kernels as integral operators, as one may show (see \cite[Thm. 269]{deGosson:2011}) that 
\begin{equation*}
  \inner{S}{T}_{\HS} = \tr(ST^*).
\end{equation*}

\subsection{A space of nuclear operators}
Both Hilbert-Schmidt and trace class operators will often be too large spaces for our purposes. We therefore introduce a Banach subspace of $\HS$ more adapted to the needs of time-frequency analysis. Let $v$ be a submultiplicative weight function. The space we will need is the space $\mathcal{N}(L^2;M^1_v)$ consisting of all \textit{nuclear} operators $T:L^2(\Rd)\to M^1_v(\Rd)$. An operator $T:L^2(\Rd)\to M^1_v(\Rd)$ is said to be nuclear \cite{Ryan:2002} if it has an expansion of the form
\begin{equation} \label{eq:nucleardefinition}
  T=\sum_{n=1}^\infty \phi_n \otimes \xi_n,
\end{equation}
where $\phi\otimes \psi$ denotes the rank-one operator 
\begin{equation*}
\phi\otimes \psi (\xi)=\inner{\xi}{\psi}_{L^2}\phi
\end{equation*}
 and $\sum_{n=1}^\infty \|\phi_n\|_{M^1_v} \|\xi_n\|_{L^2}<\infty$. The space $\mathcal{N}(L^2;M^1_v)$ becomes a Banach space with norm given by 
\begin{equation} \label{eq:nuclearnorm}
  \|T\|_{\mathcal{N}}:= \inf \left\{ \sum_{n=1}^\infty \|\phi_n\|_{M^1_v} \|\xi_n\|_{L^2} \right\},
\end{equation}
where the infimum is taken over all decompositions as in \eqref{eq:nucleardefinition}. 
It can be shown that if $\phi \in M^1_v(\Rd)$ and $\psi \in L^2(\Rd)$, then 
\begin{equation} \label{eq:nuclearnormrankone}
  \|\phi\otimes \psi\|_{\mathcal{N}} = \|\phi\|_{M^1_v} \|\psi\|_{L^2},
\end{equation}
hence the expansion in \eqref{eq:nucleardefinition} converges absolutely in $\nuc$. Using the expansion in \eqref{eq:nucleardefinition} it is straightforward to check that the inclusion of $\nuc$ into $\mathcal{L}(L^2;M^1_v)$ is continuous, i.e. 
\begin{equation} \label{eq:opnormvsnuclear}
  \|T\|_{\mathcal{L}(L^2;M^1_v)} \leq \|T\|_{\mathcal{N}},
\end{equation}
 and that if $S\in \nuc$, $T\in \bo$ and $R\in \mathcal{L}(M^1_v)$, then $RST\in \nuc$.

We will need the following simple property.
\begin{lem} \label{lem:translatenuclear}
	Let $T\in \nuc$ for a submultiplicative weight function $v$ and let $z\in \Rdd$. Then $\pi(z)T\pi(z)^*\in \nuc$ with $\|\pi(z)T\pi(z)^*\|_{\mathcal{N}}\leq v(z) \|T\|_{\mathcal{N}}$. 
\end{lem}
\begin{proof}
	If $T$ has an expansion 
\begin{equation*}
  T=\sum_{n=1}^\infty \phi_n \otimes \xi_n
\end{equation*}
 where $\sum_{n=1}^\infty \|\phi_n\|_{M^1_v} \|\xi_n\|_{L^2}<\infty$, then 
\begin{equation*}
  \pi(z)T\pi(z)^*=\sum_{n=1}^\infty \pi(z)\phi_n \otimes \pi(z)\xi_n,
\end{equation*}
 and $$\sum_{n=1}^\infty \|\pi(z)\phi_n\|_{M^1_v} \|\pi(z)\xi_n\|_{L^2}\leq v(z) \sum_{n=1}^\infty \|\phi_n\|_{M^1_v} \|\xi_n\|_{L^2}$$ by \eqref{eq:tfbound} and the fact that $\pi(z)$ is unitary on $L^2(\Rd)$. The norm inequality then follows from the definition \eqref{eq:nuclearnorm} of the nuclear norm.
\end{proof}
To be more precise, the class of operators we will be interested in are those $S\in \bo$ such that $S^*\in \nuc$, where $S^*$ is the Hilbert space adjoint of $S$.
 We can give a much more concrete description of  this condition by noting that if $\sum_{n=1}^\infty \|\xi_n\|_{L^2} \|\phi_n\|_{M^1_v}<\infty$, then $S^*=\sum_{n=1}^\infty \xi_n \otimes \phi_n$ if and only if $S=\sum_{n=1}^\infty \phi_n \otimes \xi_n$. Hence \eqref{eq:nucleardefinition} gives that $S^*\in \nuc$ if and only if 
\begin{equation} \label{eq:nuclearexpansion}
  S=\sum_{n=1}^\infty \xi_n \otimes \phi_n
\end{equation}
 with $\sum_{n=1}^\infty \|\xi_n\|_{L^2} \|\phi_n\|_{M^1_v}<\infty$.
Abusing notation slightly, we will write $S^*\in \nuc$ to denote that $S\in \bo$ and $S^*\in \nuc$.

\begin{lem} \label{lem:nuclerproperties}
	Let $S^*\in \nuc$ for a submultiplicative weight function $v$. Then $S$ extends to a bounded operator $\tilde{S}:M^\infty_{1/v}(\Rd) \to L^2(\Rd)$ with $\|\tilde{S}\|_{\mathcal{L}(M^\infty_{1/v},L^2)}\leq \|S^*\|_{\nuc}$ by defining 
\begin{equation} \label{eq:extension}
  \inner{\tilde{S}\psi}{\phi}_{L^2}:= \inner{\psi}{S^*\phi}_{M^\infty_{1/v},M^1_v} \quad \text{ for } \psi\in M^\infty_{1/v}(\Rd), \phi \in L^2(\Rd). 
\end{equation}
Furthermore, given an expansion of $S$ of the form \eqref{eq:nuclearexpansion}, this extension satisfies
\begin{equation} \label{eq:extendedexpansion}
  \tilde{S}(\psi)=\sum_{n=1}^\infty \inner{\psi}{\phi_n}_{M^\infty_{1/v},M^1_v} \xi_n,
\end{equation}
where the sum converges absolutely in $L^2(\Rd).$
\end{lem}

\begin{proof}
The definition \eqref{eq:extension} simply means that $\tilde{S}$ is the Banach space adjoint of $S^*:L^2(\Rd)\to M^1_v(\Rd)$, hence $\tilde{S}$ is well-defined. Since for $\psi  \in L^2(\Rd)$ we have $$\inner{S\psi}{\phi}_{L^2}=\inner{\psi}{S^*\phi}_{L^2}=\inner{\psi}{S^*\phi}_{M^\infty_{1/v},M^1_v},$$ we see that $\tilde{S}$ extends $S$. The absolute convergence of the sum in \eqref{eq:extendedexpansion} follows directly from \eqref{eq:nuclearexpansion}. To show that the decomposition into rank-one operators still holds for $\tilde{S}$, we need to show that 
\begin{equation*}
  \inner{\sum_{n=1}^\infty \inner{\psi}{\phi_n}_{M^\infty_{1/v},M^1_v} \xi_n}{\phi}_{L^2}= \inner{\psi}{S^*\phi}_{M^\infty_{1/v},M^1_v} \quad \text{ for } \psi\in M^\infty_{1/v}(\Rd), \phi \in L^2(\Rd),
\end{equation*}
which is a straightforward calculation using the expansion of $S^*$ in \eqref{eq:nucleardefinition} and the fact that all expansions converge absolutely in an appropriate Banach space, so that we may take the duality brackets inside the sum. The details are left for the reader. 
\end{proof}
In what follows we will simply denote the extension $\tilde{S}$ by $S$. Note that if $S^*\in \nuc$, $R\in \bo$ and $T^*\in \mathcal{L}(M^1_v)$, then $(RST)^* \in \nuc$, as follows from using \eqref{eq:nuclearexpansion}.

The fact that we use the Hilbert space $L^2(\Rd)$ is not strictly necessary. We could have considered any separable Hilbert space $\mathcal{H}$, and required that $S\in \mathcal{L}(L^2,\mathcal{H})$ with $S^*\in \mathcal{N}(\mathcal{H},M^1_v)$. The result above would still hold, as would the main result of this paper. Our reason for considering $\mathcal{H}=L^2(\Rd)$ is that it gives us easier access to non-trivial examples, as it allows us to formulate our results in terms of integral operators as we explain in detail in the next subsection. 

\subsubsection{The projective tensor product}

The theory of nuclear operators is closely related to the projective tensor product of Banach spaces, as explained for instance in \cite{Ryan:2002}, which leads to a useful connection to integral operators. Abstractly, the projective tensor product $X\hat{\otimes}Y$  of two Banach spaces $X,Y$ is the completion of the algebraic tensor product $X\otimes Y$ with respect to the norm 
\begin{equation*}
\|u\|_{X\hat{\otimes}Y}=\inf \left\{\sum_{n=1}^N \|x_n\|_{X} \|y_n\|_{Y} :u=\sum_{k=n}^N x_n\otimes y_n \right\}.
\end{equation*}
One can show (see \cite[Prop. 2.8]{Ryan:2002}) that $X\hat{\otimes}Y$ consists precisely of elements $\sum_{n=1}^\infty x_n\otimes y_n$ such that $\sum_{n=1}^\infty \|x_n\|_{X} \|y_n\|_{Y} <\infty.$ 

When $X$ and $Y$ are function spaces on $\Rd$, which is the case we will consider, we identify the elementary tensors $x\otimes y$ for $x\in X$ and $y\in Y$ with the function $x\otimes y(s,t)=x(s)y(t)$. For instance, we identify $L^2(\Rd)\hat{\otimes}L^2(\Rd)$ with all functions $\Psi\in L^2(\Rdd)$ such that $\Psi(s,t)=\sum_{n=1}^\infty \xi_n(s) \psi_n(t)$ with $\sum_{n=1}^\infty \|\xi_n\|_{L^2} \|\psi_n\|_{L^2}<\infty$. 

Now assume that the integral kernel $k_T$ of $T\in \HS$ belongs to $X\hat{\otimes} Y$ for Banach function spaces $X,Y\subset L^2(\Rd)$. By definition, this means that we have a decomposition 
\begin{equation*}
k_T(s,t)=\sum_{n=1}^\infty x_n(s)y_n(t)
\end{equation*}
with $\sum_{n=1}^\infty \|x_n\|_X \|y_n\|_Y<\infty$. A simple calculation then shows that
\begin{equation*}
T= \sum_{n=1}^\infty x_n\otimes \overline{y_n},
\end{equation*}
where $x_n\otimes \overline{y_n}$ now denotes a rank-one operator.
Hence\footnote{Since $L^2(\Rd)$ and all other function spaces we consider are invariant under complex conjugation, we need not pay any attention to the fact that $\overline{y_n}$ appears in place of $y_n$.} if we apply this to $X=M^1_v(\Rd)$ and $Y=L^2(\Rd)$, we see that $k_T\in M^1(\Rd)\hat{\otimes}L^2(\Rd)$ is equivalent to $T$ having an expansion of the form \eqref{eq:nucleardefinition} -- i.e. $k_T\in M^1(\Rd)\hat{\otimes}L^2(\Rd)$ if and only if $T\in \nuc$.
\begin{rem}
The map $k_T \mapsto T$ is actually a Banach space isomorphism from $M^1(\Rd)\hat{\otimes}L^2(\Rd)$ to $\nuc$. Surjectivity and boundedness follow from above. Injectivity is not too difficult to show in this case, but for more general  Banach spaces $X$ and $Y$ the injectivity of the natural map from $X \hat{\otimes} Y^*$ onto $\mathcal{N}(Y,X)$ boils down to the approximation property for Banach spaces \cite[Cor. 4.8]{Ryan:2002}.
\end{rem}
The slightly awkward condition $T^* \in \nuc$ may similarly be reformulated as requiring $k_T\in L^2(\Rd)\hat{\otimes} M^1(\Rd)$, this is essentially the content of \eqref{eq:extendedexpansion}. This condition cannot be reformulated as nuclearity of $T$, which is why we have opted for phrasing it as $T^* \in \nuc$. We also mention that there is a natural isomorphism $L^2(\Rd)\hat{\otimes}M^1_v(\Rd) \cong M^1_v(\Rd)\hat{\otimes} L^2(\Rd)$ extending the map $\xi\otimes \phi \mapsto \phi \otimes \xi$ for $\xi \in L^2(\Rd),\phi \in M^1_v(\Rd).$

Formulating our assumption on $T$ by requiring $k_T$ to belong to some projective tensor product makes it possible to relate $T^*\in \nuc$ to other spaces of operators. For instance, we may identify the trace class operators $\mathcal{S}$ as the operators $S\in \HS$ such that $k_S$ belongs to the projective tensor product $L^2(\Rd)\hat{\otimes} L^2(\Rd)$, which clearly contains $M^1_v(\Rd)\hat{\otimes} L^2(\Rd)$ as a subset since $M^1_v(\Rd)\hookrightarrow L^2(\Rd)$. 

Finally, those $T\in \HS$ with kernel $k_T$ in $M^1_v(\Rd)\hat{\otimes} M^1_v(\Rd)\subset M^1(\Rd)\hat{\otimes} L^2(\Rd)$ have also been studied recently in \cite{Skrettingland:2019}, where this space of operators is denoted by $\beauty_{v\otimes v}$. It follows by \cite[Thm. 5]{Balazs:2019} that  
\begin{equation} \label{eq:projtensorfeichtinger}
M^1_v(\Rd)\hat{\otimes} M^1_v(\Rd)= M^1_{v\tilde{\otimes}v}(\Rdd),
\end{equation}
with equivalent norms, where $v\tilde{\otimes}v(x_1,x_2,\omega_1,\omega_2)=v(x_1,\omega_1)\cdot v(x_2,\omega_2)$. The particular case $\beauty := \beauty_{1\otimes 1}$ corresponding to $v\equiv 1$ has been studied in several other sources, see for instance \cite{Feichtinger:1998,Feichtinger:2018}. 

We summarize this discussion, which essentially amounts to prodding the definitions in various ways, in a proposition.

\begin{prop} \label{prop:projectivetensor}
	Given $T\in \HS$ and a submultiplicative weight function $v$. Then
	\begin{align*}
T\in \nuc &\iff k_T \in M^1_v(\Rd)\hat{\otimes}L^2(\Rd), \\
T^* \in \nuc &\iff k_T \in L^2(\Rd)\hat{\otimes}M^1_v(\Rd).
\end{align*}
At the level of $k_T$ we have the chain of inclusions 
{\small $$M^1_v(\Rd)\hat{\otimes}M^1_v(\Rd)\subset L^2(\Rd)\hat{\otimes}M^1_v(\Rd), M^1_v(\Rd)\hat{\otimes} L^2(\Rd)\subset L^2(\Rd)\hat{\otimes} L^2(\Rd)\subset L^2(\Rdd),$$} which at the operator level leads to the inclusions $$ \beauty_{v\otimes v}\subset \nuc \subset \mathcal{S}\subset \HS.$$
The same inclusion holds when $\nuc$ is replaced by the set of operators $T$ such that $T^*\in \nuc$.
\end{prop}

\subsection{Examples of nuclear operators}
The connection to the projective tensor product allows us to write down some examples of $S^*\in \nuc$. 

\begin{exmp} \label{exmp:schwartz}
	The \textit{Schwartz operators} $\mathfrak{S}$ are those integral operator $T_k$ on $L^2(\Rd)$ such that $k\in \mathscr{S}(\Rdd)$\cite{Keyl:2015}. If the submultiplicative weight $v$ grows at most polynomially, then so does the weight function $v\tilde{\otimes}v(x_1,x_2,\omega_1,\omega_2)=v(x_1,\omega_1)\cdot v(x_2,\omega_2)$ on $\R^{4d}$, hence we know that $\mathscr{S}(\Rdd)\hookrightarrow M^1_{v\tilde{\otimes}v}(\Rdd)\cong M^1_v(\Rd) \hat{\otimes} M^1_v(\Rd)$. It follows that $\mathfrak{S}\subset \beauty_{v\otimes v}\subset \nuc$. It is also straightforward to check that $S\in \mathfrak{S} \iff S^*\in \mathfrak{S}.$
\end{exmp}

\begin{exmp}[The Feichtinger algebra and the inner kernel theorem] 
	By Proposition \ref{prop:projectivetensor} we know that $\beauty_{v\otimes v}\subset \nuc$, where $T\in \beauty_{v\otimes v}$ if $k_T\in M^1_{v\tilde{\otimes}v}(\Rdd)\cong M^1_v(\Rd)\hat{\otimes} M^1_v(\Rd)$. This class of operators was recently studied in \cite{Skrettingland:2019}, where the reader may find a proof that $T$ belongs to this space if and only if its Hilbert space adjoint $T^*$ does.

The unweighted case $k_T\in M^1(\Rd)\hat{\otimes}M^1(\Rd) = M^1(\Rdd)$ has been studied by several sources \cite{Feichtinger:1998,Kozek:2006,Feichtinger:2018,Skrettingland:2019qhal}. We mention in particular that \cite{Feichtinger:1998,Feichtinger:2018} give a characterization of such operators that is independent of their kernel as an integral operator: Given $T\in \HS$, $\kernel_T \in M^1(\Rdd)$ if and only if $T$ extends to a bounded map $M^\infty(\Rd)\to M^1(\Rd)$ sending weak* convergent sequences to norm-convergent sequences. 
\end{exmp}

We now consider finite rank operators. By choosing $S$ of the form in this example, we will be able to recover Theorem \ref{thm:windowindependent} from our main result.

 \begin{exmp}[Finite rank operators] \label{exmp:finiterank}
 For $N\in \mathbb{N}$, consider $\{\phi_n\}_{n=1}^N\subset M^1_v(\Rd).$ Let $\{\xi_n\}_{n=1}^N$ be an orthonormal set in $L^2(\Rd)$. If we define $S=\sum_{n=1}^N \xi_n \otimes \phi_n$, we clearly have $S^* \in \nuc$. This $S$ is just a convenient way of storing the functions $\phi_n$ in an operator -- by applying $S$ to $\xi_m$ for $1\leq m \leq N$ we recover $\phi_m$.   
 
 \end{exmp}

\subsubsection{Localization operators} \label{sec:locop}

We also have some methods for producing new examples of operators in $\nuc$ from known examples. As $\nuc$ is a normed space we may of course take linear combinations, but a more interesting method is to use the quantum convolutions introduced by Werner\cite{Werner:1984}. Given $f\in L^1(\Rdd)$ and a trace class operator $S\in \tco$, the \textit{convolution} of $f$ with $S$ is defined to be the trace class operator $f\star S$ given by the Bochner integral
\begin{equation} \label{eq:opconv}
  f\star S:= \int_{\Rdd} f(z) \pi(z)S\pi(z)^* \ dz.
\end{equation}
 In particular, if we pick $S$ to be a rank-one operator $\varphi_2 \otimes \varphi_1$ for $\varphi_1,\varphi_2\in L^2(\Rd)$, we find that $$f\star (\varphi_2 \otimes \varphi_1)=\mathcal{A}^{\varphi_1,\varphi_2}_f,$$ where $\mathcal{A}^{\varphi_1,\varphi_2}_f$ is the \textit{time-frequency localization operator} \cite{Daubechies:1988,Cordero:2003} given by
\begin{equation*}
  \mathcal{A}^{\varphi_1,\varphi_2}_f(\psi)=\int_{\Rdd} f(z) V_{\varphi_1}\psi(z) \pi(z) \varphi_2 \ dz \quad \text{ for } \psi \in L^2(\Rd). 
\end{equation*}

\begin{prop} 
	If $S\in \nuc$ and $f\in L^1_v(\Rdd)$, then $f\star S \in \nuc$ with $\|f\star S\|_{\mathcal{N}}\leq \|f\|_{L^1_v} \|S\|_{\mathcal{N}}$. In particular, if $\varphi_1\in L^2(\Rd)$ and $\varphi_2\in M^1_v(\Rd)$, then $\mathcal{A}_f^{\varphi_1,\varphi_2}\in \nuc$, with $\|\mathcal{A}_f^{\varphi_1,\varphi_2}\|_{\mathcal{N}}\leq \|f\|_{L^1_v} \|\varphi_1\|_{L^2} \|\varphi_2\|_{M^1_v} $.
\end{prop}

\begin{proof}
	By definition,
\begin{equation*}
  f\star S = \int_{\Rdd} f(z) \pi(z)S\pi(z)^* \ dz.
\end{equation*}
This integral converges as a Bochner integral in $\nuc$, as Lemma \ref{lem:translatenuclear} gives
\begin{equation*}
  \int_{\Rdd}  \|f(z)\pi(z)S\pi(z)^*\|_{\mathcal{N}} \ dz \leq \int_{\Rdd} |f(z)| v(z) \|S\|_{\mathcal{N}} \ dz = \|S\|_{\mathcal{N}} \|f\|_{L^1_v}.
\end{equation*}
The result for $\mathcal{A}_f^{\varphi_1,\varphi_2}$ follows by $\mathcal{A}_f^{\varphi_1,\varphi_2}=f\star (\varphi_2 \otimes \varphi_1)$ and \eqref{eq:nuclearnormrankone}.
\end{proof}
It is easy to check that the Hilbert space adjoint of $f\star S$ is $\overline{f}\star S^*$. Hence we immediately obtain the following.

 \begin{cor} \label{cor:locopbound} 
	If $S^*\in \nuc$ and $f\in L^1_v(\Rdd)$, then $(f\star S)^* \in \nuc$ with $\|(f\star S)^*\|_{\mathcal{N}}\leq \|f\|_{L^1_v} \|S^*\|_{\mathcal{N}}$. In particular, if $\varphi_1\in M^1_v(\Rd)$ and $\varphi_2\in L^2(\Rd)$, then $\left(\mathcal{A}_f^{\varphi_1,\varphi_2}\right)^*\in \nuc$, with $\|\left(\mathcal{A}_f^{\varphi_1,\varphi_2}\right)^*\|_{\mathcal{N}}\leq \|f\|_{L^1_v} \|\varphi_1\|_{M^1_v} \|\varphi_2\|_{L^2} $.
\end{cor}

\subsubsection{Underspread operators}
	When operators between function spaces are used to model communication channels, the resulting operators will typically be (at least approximately) \textit{underspread} \cite{Strohmer:2006}. An underspread operator $T\in \HS$ is of the form 
\begin{equation} \label{eq:spreading}
  T=\int_{\Rdd} F(x,\omega) e^{-i\pi x\cdot \omega} \pi(x,\omega) \ dxd\omega,
\end{equation}
where the support of $F$ is contained in $[-\tau,\tau]^d\times [-\nu,\nu]^d$ for $4\tau\nu <1.$ The function $F(x,\omega)e^{-i\pi x\cdot \omega} $ is called the spreading function of $T$, and one can show that any $T\in \HS$ has a spreading function in $L^2(\Rdd)$, as long as the integral in \eqref{eq:spreading} is interpreted appropriately\cite{Feichtinger:1998}. In quantum harmonic analysis the spreading function is considered a  Fourier transform of the operator\cite{Werner:1984}. The next lemma shows that underspread \textit{trace class} operators belong to $\beauty$, i.e. have integral kernel in $M^1(\Rdd)\subset \mathcal{N}(L^2;M^1)$. This is an operator-version of the well-known fact that band-limited integrable functions belong to $M^1(\Rd)$ \cite[Cor. 3.2.7]{Feichtinger:1998a}. The proof is moved to an appendix, as it requires the introduction of several results from quantum harmonic analysis that will not be needed later in the paper. 
\begin{prop} \label{prop:underspread}
	If the spreading function of $T\in \tco$ has compact support, then $T\in \beauty\subset \mathcal{N}(L^2;M^1)$. 
\end{prop}

\section{Time-frequency analysis with operators as windows} \label{sec:tfa}

A fundamental object in time-frequency analysis is the short-time Fourier transform (STFT) $V_\phi\psi$ with window $\phi$. The goal of this section is the define an STFT where the window $\phi$ is replaced by an operator $S$, and to show that the basic properties of the STFT remain true for this generalized STFT. 

As a first step, we will need the Hilbert space $L^2(\Rdd;L^2)$ of equivalence classes of strongly Lebesgue measurable $\Psi:\Rdd \to L^2(\Rd)$ such that
\begin{equation*}
  \|\Psi\|_{L^2(\Rdd;L^2)} := \left(\int_{\Rdd} \|\Psi(z)\|_{L^2}^2 \ dz \right)^{1/2}<\infty,
\end{equation*}
with inner product 
 \begin{equation*}
  \inner{\Psi}{\Phi}_{L^2(\Rdd;L^2)}=\int_{\Rdd} \inner{\Psi(z)}{\Phi(z)}_{L^2} \ dz.
\end{equation*}
The equivalence relation on $L^2(\Rdd;L^2)$ is that $\Psi \sim \Phi$ if $\Psi(z)=\Phi(z)$ as elements of $L^2(\Rd)$ for a.e. $z\in \Rdd.$

We then define a version of the short-time Fourier transform with operators as windows. For $S\in \HS$ and $\psi\in L^2(\Rd)$ we let 
\begin{equation*}
  \opstft_S(\psi)(z)=S\pi(z)^*\psi \quad \text{ for } z\in \Rdd.
\end{equation*}

We obtain a generalization of Moyal's identity. It shows that $\opstft_S$ is a linear isometry from $L^2(\Rd)$ to $L^2(\Rdd;L^2)$.

\begin{lem} \label{lem:genmoyal}
	Let $S\in \HS$ and $\psi\in L^2(\Rd)$. Then $\opstft_S(\psi)\in L^2(\Rdd;L^2)$ and
\begin{equation*} 
   \|\opstft_S(\psi)\|^2_{L^2(\Rdd,L^2)}=\int_{\Rdd} \|S\pi(z)^*\psi\|_{L^2}^2 \ dz = \|S\|_{\HS}^2 \|\psi\|_{L^2}^2.
\end{equation*}
\end{lem}
\begin{proof}
	We may rewrite $$\|S\pi(z)^*\psi\|_{L^2}^2=\inner{\pi(z)S^*S\pi(z)^*\psi}{\psi}_{L^2}=\tr(\pi(z)S^*S\pi(z)^*(\psi\otimes \psi)).$$ The result therefore follows by \cite[Lem. 4.1]{Luef:2018c}, which states that $$\int_{\Rdd} \tr(\pi(z)R\pi(z)^*T) \ dz=\tr(R)\tr(T)$$ for trace class operators $R,T$ on $L^2(\Rd)$; pick $R=S^*S$ and $T=\psi \otimes \psi$ and note that $\tr(S^*S)=\|S\|_{\HS}^2$ and $\tr(\psi\otimes \psi)=\|\psi\|_{L^2}^2$
\end{proof}

\begin{exmp}
	To see that $\opstft_S$ actually generalizes the usual STFT, consider $\phi\in L^2(\Rd)$ and let $\xi\in L^2(\Rd)$ be any function satisfying $\|\xi\|_{L^2}=1$. Then let $S=\xi \otimes \phi.$ For any $\psi \in L^2(\Rd)$ we then have $$\opstft_S(\psi)(z)=S\pi(z)^*\psi=\inner{\pi(z)^*\psi}{\phi}_{L^2}\xi=V_{\phi}\psi(z) \cdot \xi,$$ which contains precisely the same information as $V_\phi\psi(z)$ given that we know $\xi.$ In particular, it is easy to show that $\|S\|_{\HS}^2=\|\phi\|_{L^2}^2$ and $\|\opstft_S(\psi)\|_{L^2(\Rdd;L^2)}^2=\|V_\phi \psi\|_{L^2(\Rdd)}$, so Lemma \ref{lem:genmoyal} reduces to Moyal's identity in this case.
\end{exmp}

\begin{rem}
	Strong measurability of $\opstft_S(\psi)$ is always satisfied: since $L^2(\Rd)$ is separable, the Pettis measurability theorem \cite[Thm. 1.1.20]{Hytonen:2016} ensures that strong measurability follows from weak measurability. Weak measurability means that for each $\phi\in L^2(\Rd)$, the map $$z\mapsto  \inner{\opstft_S(\psi)(z)}{\phi}_{L^2}$$ is Lebesgue measurable. We may rewrite $\inner{\opstft_S(\psi)(z)}{\phi}_{L^2}=\inner{S\pi(z)^*\psi}{\phi}_{L^2}=V_{S^*\phi}\psi(z)$. It is well-known that the STFT $z\mapsto V_\xi\psi(z)$ is continuous for any $\xi \in L^2(\Rd)$, in particular for $\xi=S^*\phi$, hence the map is Lebesgue measurable.
\end{rem}
 We then define for $\Psi \in L^2(\Rdd;L^2)$ a function $\opstft_S^*(\Psi)$ on $\Rd$ by the $L^2(\Rd)$-valued  integral
\begin{equation} \label{eq:synthesisintegral}
  \opstft_S^*(\Psi)=\int_{\Rdd} \pi(z)S^* \Psi(z) \ dz.
\end{equation}
The integral \eqref{eq:synthesisintegral} is interpreted in a weak sense: we will see that
\begin{equation}\label{eq:weaksensecondition}
\left|\int_{\Rdd} \inner{\pi(z)S^* \Psi(z)}{\phi}_{L^2} \ dz \right|\lesssim \|\phi\|_{L^2} \quad \text{ for any } \phi\in L^2(\Rd),
\end{equation}
so it follows from the Riesz representation theorem for Hilbert spaces that there must exist an element in $L^2(\Rd)$, which we denote by $\int_{\Rdd} \pi(z)S^* \Psi(z) \ dz,$ such that for any $\phi \in L^2(\Rd)$
\begin{equation} \label{eq:adjointweakdef}
  \inner{\int_{\Rdd} \pi(z)S^* \Psi(z) \ dz}{\phi}_{L^2}=\int_{\Rdd} \inner{\pi(z)S^* \Psi(z)}{\phi}_{L^2} \ dz.
\end{equation}
The next lemma shows that the integral in \eqref{eq:synthesisintegral} is well-defined in this sense.

\begin{lem} \label{lem:continuous}
	Let $S \in \HS$. 
	\begin{enumerate}[(a)]
		\item  Equation \eqref{eq:adjointweakdef} defines $\opstft_S^*(\psi)=\int_{\Rdd} \pi(z)S^* \Psi(z) \ dz$ as an element of $L^2(\Rd)$, and $\opstft_S^*:L^2(\Rdd;L^2)\to L^2(\Rd)$ defines a bounded operator that is the Hilbert space adjoint of $\opstft_S$.
		\item The composition $\opstft_S^*\opstft_S$ is $\|S\|_{\HS}^2$ times the identity operator on $L^2(\Rd)$.
	\end{enumerate}
\end{lem}
\begin{proof}
  Let $\Psi\in L^2(\Rdd;L^2)$ and let $\phi\in L^2(\Rd)$. We need to show \eqref{eq:weaksensecondition}, as mentioned \eqref{eq:adjointweakdef} then defines an element $\int_{\Rdd} \pi(z)S^* \Psi(z) \ dz$ of $L^2(\Rd)$ by Riesz' representation theorem. We find that
 \begin{align*}
 \left| \int_{\Rdd} \inner{\pi(z)S^*\Psi(z)}{\phi}_{L^2} \ dz \right| &= \left|\int_{\Rdd} \inner{\Psi(z)}{S\pi(z)^*\phi}_{L^2} \ dz \right| \\
 &= |\inner{\Psi}{\opstft_S(\phi)}_{L^2(\Rd;L^2)}| \\
 &\leq \|\Psi\|_{L^2(\Rdd;L^2)} \|\opstft_S(\phi)\|_{L^2(\Rdd;L^2)}\\
 &=\|\Psi\|_{L^2(\Rdd;L^2)} \|S\|_{\HS} \|\phi\|_{L^2} 
\end{align*}
by Lemma \ref{lem:genmoyal}. It is clear that $\opstft_S^*$ is linear, and the estimate also shows that it is bounded from $L^2(\Rd;L^2(\Rd))$ to $L^2(\Rd)$. A simple calculation shows that it is the adjoint of $\opstft_S.$  The second part states that $$\int_{\Rdd} \pi(z)S^*S\pi(z)^*\psi \ dz=\|S\|_{\HS}^2 \psi \quad \text{ for any } \psi \in L^2(\Rd),$$ which is part (c) of \cite[Prop. 3.3]{Werner:1984}.
\end{proof}

\section{Equivalent norms for modulation spaces} \label{sec:mainresult}

The generalized Moyal identity in Lemma \ref{lem:genmoyal} shows that the norm of $\opstft_S(\psi)$ in $L^2(\Rdd;L^2)$ is equivalent to the norm of $\psi$ in $L^2(\Rd)$. We will now generalize Theorem \ref{thm:windowindependent} by showing that if $S$ satisfies some extra assumptions, the same is true if $L^2(\Rd)$ is replaced by $M^{p,q}_m(\Rd)$ and $L^2(\Rdd;L^2)$ is replaced by $L^{p,q}_m(\Rdd;L^2)$, where $1\leq p,q \leq \infty$ and $m$ is some $v$-moderate weight. As before, $v$ always denotes a submultiplicative weight function on $\Rdd$. 

We start by defining $L^{p,q}_m(\Rdd;L^2)$. For $1\leq p,q \leq \infty$ and any $v$-moderate weight $m$, the Banach space $L^{p,q}_m(\Rdd;L^2)$ consists of the equivalence classes of strongly Lebesgue measurable functions  $\Psi : \Rdd \to L^2(\Rd)$ such that 
\begin{equation*}
  \|\Psi\|_{L^{p,q}_m(\Rdd;L^2)} := \left( \int_{\Rd} \left(\int_{\Rd} \|\Psi(x,\omega)\|_{L^2}^p m(x,\omega)^p \ dx \right)^{q/p} d\omega\right)^{\frac{1}{q}}<\infty,
\end{equation*}
where $\Phi\sim \Psi$ if $\Psi(z)=\Phi(z)$ for a.e. $z\in \Rdd$. When $p=\infty$ or $q=\infty$ the definition is modified in the usual way by replacing integrals by essential supremums. 

With this definition in place, we are ready to state our main result. 

\begin{thm} \label{thm:continuous}
	Let $0\neq S\in \HS$ such that $S^*\in \nuc$. For any $1\leq p,q \leq \infty$ and $v$-moderate weight $m$, we have
	\begin{equation*}
    C_{lower}\cdot\|\psi\|_{M^{p,q}_m}\leq \|\opstft_S(\psi)\|_{L^{p,q}_m(\Rdd;L^2)} \leq C_{upper}\cdot  \|\psi\|_{M^{p,q}_m}
\end{equation*}
with 
\begin{align*}
C_{lower}&=\|S\|_{\HS}^2\cdot \left(C_v^m\cdot  \|S^*\|_{\mathcal{N}}\cdot  \|\varphi_0\|_{M^1_v} \right)^{-1}, \\
C_{upper}&=C_v^m \cdot \|S^*\|_{\mathcal{N}}.
\end{align*}
\end{thm}

Our proof will follow the same structure as the usual proof that $M^{p,q}_m$ is independent of the window function \cite{Grochenig:2001}: we will show that $\opstft_S$ is bounded from $M^{p,q}_m(\Rd)$ to $L^{p,q}_m(\Rdd;L^2)$ and that $\opstft_S^*$ is bounded from $L^{p,q}_m(\Rdd;L^2)$ to $M^{p,q}_m(\Rd)$.

 Before we start, we make sure that there is no ambiguity in interpreting $$\opstft_S(\psi)(z)=S\pi(z)^*\psi$$ even when $\psi\in M^\infty_{1/v}(\Rdd)$. First note that as $\pi(z)$ is bounded on $M^1_v(\Rd)$ by \eqref{eq:tfbound}, we may extend $\pi(z)$ to a bounded operator on $M^\infty_{1/v}(\Rd)$ by defining
\begin{equation} \label{eq:tfdual}
  \inner{\pi(z)\psi}{\phi}_{M^\infty_{1/v},M^1_v}:=\inner{\psi}{\pi(z)^*\phi}_{M^\infty_{1/v},M^1_v} \quad \text{ for } \psi \in M^\infty_{1/v}(\Rd),\phi \in M^1_v(\Rd).
\end{equation}
As $\pi(x,\omega)^*=e^{-2\pi i x \cdot \omega}\pi(-x,\omega)$, $\pi(z)^*$ is also bounded on $M^\infty_{1/v}(\Rd)$. Therefore
 $$\opstft_S(\psi)(z)=S\pi(z)^*\psi$$ makes sense by Lemma \ref{lem:nuclerproperties}, as $S$ extends to a bounded operator from $M^\infty_{1/v}(\Rd)$ to $L^2(\Rd)$ -- hence $S\pi(z)^*\psi$ is a well-defined element of $L^2(\Rd)$.

\begin{lem} \label{lem:analysisbound}
	Let $m$ be a $v$-moderate weight. For any $1\leq p,q\leq \infty$, $\opstft_S$ is a bounded, linear map from $M^{p,q}_m(\Rd)$ to $L^{p,q}_m(\Rdd;L^2)$ with $\|\opstft_S(\psi)\|_{L^{p,q}_m(\Rdd;L^2)}\leq C_{v}^m\cdot \|S^*\|_{\mathcal{N}}\cdot \|\psi\|_{M^{p,q}_m}$.
\end{lem}

\begin{proof}
Throughout the proof we will use the expansion in \eqref{eq:extendedexpansion} to write
\begin{equation*} 
  S=\sum_{n=1}^\infty \xi_n \otimes \phi_n
\end{equation*}
 with $\sum_{n=1}^\infty \|\xi_n\|_{L^2} \|\phi_n\|_{M^1_v}<\infty$.
	Assume that $\psi \in M^{p,q}_m(\Rd)$. Then 
	\begin{equation*}
S\pi(z)^*\psi=\sum_{n=1}^\infty \inner{\pi(z)^*\psi}{\phi}_{M^\infty_{1/v},M^1_v} \xi_n=\sum_{n=1}^\infty V_{\phi_n}\psi(z) \xi_n.
\end{equation*}
	This implies that 
	\begin{equation*}
  \|S\pi(z)^*\psi\|_{L^2}\leq \sum_{n=1}^\infty |V_{\phi_n}\psi(z)| \cdot  \|\xi_n\|_{L^2},
\end{equation*}

hence the triangle inequality for $L^{p,q}_m(\Rdd)$ gives
\begin{align*}
  \|\opstft_S(\psi)\|_{L^{p,q}_m(\Rdd;L^2)}&\leq \left\|\sum_{n=1}^\infty |V_{\phi_n}\psi(-)| \cdot  \|\xi_n\|_{L^2} \right\|_{L^{p,q}_m(\Rdd)} \\
&\leq \sum_{n=1}^\infty \|\xi_n\|_{L^2} \left\|V_{\phi_n}\psi\right\|_{L^{p,q}_m(\Rdd)}
\end{align*}
 We then apply Proposition \ref{prop:rankonecase} to get 
\begin{equation*}
  \|\opstft_S(\psi)\|_{L^{p,q}_m(\Rdd;L^2)}\leq C_v^m \|\psi\|_{M^{p,q}_m} \sum_{n=1}^\infty \|\xi_n\|_{L^2} \|\phi_n\|_{M^1_v}.
\end{equation*}
Using the definition of $\|S^*\|_{\mathcal{N}}$ from \eqref{eq:nuclearnorm} we get that
\begin{equation*}
 \|\opstft_S(\psi)\|_{L^{p,q}_m(\Rdd;L^2)} \leq C_v^m  \|S^*\|_{\mathcal{N}}  \|\psi\|_{M^{p,q}_m}.
\end{equation*}
\end{proof}

In order to give a sensible definition of $\opstft_S^*(\Psi)$ for $\Psi \in L^{p,q}_m(\Rdd;L^2)$, we will need H\"older's inequality for the mixed-norm spaces $L^{p,q}_m(\Rdd)$ \cite{Benedek:1961,Grochenig:2001}: given $F\in L^{p,q}_m(\Rdd)$ and $G\in L^{p',q'}_{1/m}(\Rdd)$ for $1\leq p,q\leq \infty$, then $F\cdot G \in L^1(\Rdd)$ with 
\begin{equation} \label{eq:holder}
  \int_{\Rdd} |F(z)G(z)| \ dz \leq \|F\|_{L^{p,q}_m}\|G\|_{L^{p',q'}_{1/m}}.
\end{equation}

For any $\Psi\in L^{p,q}_m(\Rdd;L^2)$ we then define $\opstft_S^*(\Psi)$ as an element of $M^\infty_{1/v}(\Rd)$ by duality:
\begin{equation*}
  \inner{\opstft_S^*(\Psi)}{\phi}_{M^\infty_{1/v},M^1_v}:=\int_{\Rdd} \inner{\Psi(z)}{\opstft_S(\phi)(z)}_{L^2} \ dz \quad \text{ for all } \phi\in M^1_v(\Rd). 
\end{equation*}
To see that this actually defines a bounded linear functional on $M^1_v(\Rd)$, note that 
\begin{align*}
  \int_{\Rdd} |\inner{\Psi(z)}{\opstft_S(\phi)(z)}_{L^2}| \ dz &\leq \int_{\Rdd} \|\Psi(z)\|_{L^2} \|\opstft_S(\phi)(z)\|_{L^2} \ dz \\
&\leq \|\Psi\|_{L^{p,q}_m(\Rdd;L^2)} \|\opstft_S(\phi)\|_{L^{p',q'}_{1/m}(\Rdd;L^2)} \quad \text{ by \eqref{eq:holder}} \\
&\leq \|\Psi\|_{L^{p,q}_m(\Rdd;L^2)} C_v^m \|S^*\|_{\mathcal{N}} \|\phi\|_{M^{p',q'}_{1/m}} \quad \text{ by Lemma \ref{lem:analysisbound}} \\
&\lesssim  \|\Psi\|_{L^{p,q}_m(\Rdd;L^2)} C_v^m \|S^*\|_{\mathcal{N}} \|\phi\|_{M^{1}_{v}},
\end{align*} 
where the last inequality uses that $M^1_v(\Rd)\hookrightarrow M^{p,q}_m(\Rd)$ for all $1\leq p,q \leq \infty$ and all $v$-moderate weights $m$. The reader should observe that this definition agrees with our original definition \eqref{eq:synthesisintegral} when $\Psi \in L^2(\Rdd;L^2).$

\begin{lem} \label{lem:synthesisbound}
	Let $m$ be a $v$-moderate weight. For any $1\leq p,q\leq\infty$, the map $\opstft_S^*$
is a bounded, linear map from $L^{p,q}_m(\Rdd;L^2)$ to $M^{p,q}_m(\Rd)$, with $\|\opstft_S^*(\Psi)\|_{M^{p,q}_m}\leq \|\Psi\|_{L^{p,q}_m(\Rdd;L^2)}\cdot C_v^m\cdot \|S^*\|_{\mathcal{N}}\cdot \|\varphi_0\|_{M^1_v}$

\end{lem}
\begin{proof}
As a short preparation, 	we consider $\opstft_S(\pi(z)\phi)$ for  $\phi \in L^2(\Rd)$. By definition  $$\opstft_S(\pi(z)\phi)(z')=S\pi(z')^*\pi(z)\phi=S[\pi(z)^*\pi(z')]^*\phi.$$
With $z=(x,\omega)$ and $z'=(x',\omega')$, we find using \eqref{eq:tfcomposition} and \eqref{eq:tfadjoint} that \begin{equation} \label{eq:intertwining}
  \opstft_S(\pi(z)\phi)(z')=S[e^{2\pi i x\cdot(\omega'-\omega)}\pi(z'-z)]^*\phi=e^{2\pi i x\cdot (\omega-\omega')}\opstft_S(\phi)(z'-z).
\end{equation}

	Recall that $\varphi_0$ is the $L^2$-normalized Gaussian on $\Rd$, and that the norm on $M^{p,q}_m(\Rd)$ is given by $\|\psi\|_{M^{p,q}_m}=\|V_{\varphi_0}\psi\|_{L^{p,q}_m}.$ We therefore calculate that
\begin{align*}
  |V_{\varphi_0}(\opstft_S^*(\Psi))(z)|&=|\inner{\opstft_S^*(\Psi) }{\pi(z)\varphi_0}_{M^\infty_{1/v},M^1_v}| \\
&= \left| \int_{\Rdd} \inner{\Psi(z')}{\opstft_S(\pi(z)\varphi_0)(z')}_{L^2} \ dz' \right| \\
&\leq \int_{\Rdd} |\inner{\Psi(z')}{\opstft_S(\pi(z)\varphi_0)(z')}_{L^2}| \ dz' \\
&\leq  \int_{\Rdd} \|\Psi(z')\|_{L^2} \|\opstft_S(\pi(z)\varphi_0)(z')\|_{L^2} \ dz' \\
&=\int_{\Rdd} \|\Psi(z')\|_{L^2} \|\opstft_S(\varphi_0)(z'-z)\|_{L^2} \ dz' \quad \text{ by \eqref{eq:intertwining}.}
\end{align*}

By \cite[Prop. 11.1.3]{Grochenig:2001} the space $L^{p,q}_m(\Rdd)$ satisfies the convolution relation
\begin{equation} \label{eq:mixedyoung}
  \|F\ast G\|_{L^{p,q}_m} \leq \|F\|_{L^{p,q}_m} \|G\|_{L^1_v}
\end{equation}
for $F\in L^{p,q}_m(\Rdd)$ and $G\in L^1_v(\Rdd)$. If we let $F(z)=\|\Psi(z)\|_{L^2}$ and $G(z)=\|\opstft_S(\varphi_0)(-z)\|_{L^2}$ the calculation above states that 
\begin{equation*}
|V_{\varphi_0}(\opstft_S^*(\Psi))(z)| \leq F\ast G(z),
\end{equation*} 
which in light of \eqref{eq:mixedyoung} gives
\begin{align*}
  \|V_{\varphi_0}(\opstft_S^* \Psi)\|_{L^{p,q}_m} &\leq \|F\|_{L^{p,q}_m} \|G\|_{L^1_v} \\
  &= \|\Psi\|_{L^{p,q}_m(\Rdd;L^2)} \|\opstft_S(\varphi_0)\|_{L^{1}_v(\Rdd;L^2)} \\
&\leq \|\Psi\|_{L^{p,q}_m(\Rdd;L^2)}C_v^m \|S^*\|_{\mathcal{N}} \|\varphi_0\|_{M^1_v}  ,
\end{align*}
where we have used Lemma \ref{lem:analysisbound} in the last step. The reader should also note that $\|\opstft_S(\varphi_0)\|_{L^{1}_v(\Rdd;L^2)}=\|G\|_{L^1_v}$ is a straightforward computation, but relies on our assumption that $v(-z)=v(z).$ 
\end{proof}

Finally, we also need that the inversion formula $\opstft_S^*\opstft_S\psi =\|S\|_{\HS}^2 \psi$ from Lemma \ref{lem:continuous} remains valid on the other modulation spaces.
\begin{lem}
	Let $\psi \in M^{\infty}_{1/v}(\Rdd).$ Then $\|S\|_{\HS}^2\cdot\psi=  \opstft_S^* \opstft_S(\psi).$
\end{lem}

\begin{proof}

We need to show that for $\phi\in M^1_v(\Rd)$ we have
\begin{equation*}
  \inner{\opstft_S^* \opstft_S \psi}{\phi}_{M^\infty_{1/v},M^1_v}=\|S\|_{\HS}^2 \inner{\psi}{\phi}_{M^\infty_{1/v},M^1_v}.
\end{equation*}
As a preliminary step, we rewrite the left hand side of this expression in a way that involves explictly the action of $\psi$ as a functional: \begin{align*}
   \inner{\opstft_S^* \opstft_S(\psi)}{\phi}_{M^\infty_{1/v},M^1_v}&=\int_{\Rdd} \inner{\opstft_S(\psi)(z)}{\opstft_S(\phi)(z)}_{L^2} \ dz \\
&= \int_{\Rdd} \inner{S\pi(z)^*\psi}{S\pi(z)^*\phi}_{L^2} \ dz \\
&= \int_{\Rdd} \inner{\pi(z)^*\psi}{S^*S\pi(z)^*\phi}_{M^{\infty}_{1/v},M^1_v} \ dz \quad \text{ by \eqref{eq:extension}}\\
&= \int_{\Rdd} \inner{\psi}{\pi(z)S^*S\pi(z)^*\phi}_{M^\infty_{1/v},M^1_v} \ dz \quad \text{ by \eqref{eq:tfdual}}.
\end{align*}

Hence it suffices to show that 
\begin{equation*}
 \int_{\Rdd} \inner{\psi}{\pi(z)S^*S\pi(z)^*\phi}_{M^\infty_{1/v},M^1_v} \ dz=\|S\|_{\HS}^2 \inner{\psi}{\phi}_{M^\infty_{1/v},M^1_v}.
\end{equation*}
When $\psi \in L^2(\Rd)\subset M^\infty_{1/v}(\Rd)$, this holds by Lemma \ref{lem:continuous}. To proceed, we will use that for any $\psi\in M^\infty_{1/v}(\Rd)$ there exists a sequence $\{\psi_n\}_{n=1}^\infty$ in $L^2(\Rd)$ with $\|\psi_n\|_{M^\infty_{1/v}}\lesssim \|\psi\|_{M^\infty_{1/v}}$ such that $\psi_n$ converges to $\psi$ in the weak* topology of $M^\infty_{1/v}(\Rd)$ as $n\to \infty$; a construction of such a sequence may be found in the proof of \cite[Cor. 7]{Dorfler:2011}. Let us define 
\begin{align*}
  \Xi_n &:=  \|S\|_{\HS}^2 \inner{\psi_n}{\phi}_{M^\infty_{1/v},M^1_v} \\
&=  \int_{\Rdd} \inner{\psi_n}{\pi(z)S^*S\pi(z)^*\phi}_{M^\infty_{1/v},M^1_v} \ dz.
\end{align*}
Using the upper expression for $\Xi_n$ above, we have that $\Xi_n \to \|S\|_{\HS}^2 \inner{\psi}{\phi}_{M^\infty_{1/v},M^1_v}$ as $n\to \infty$ by the weak* convergence of $\psi_n$ to $\psi$. Using the lower expression, we find -- assuming for now that the limit may be taken inside the integral -- that 

\begin{align*}
  \lim_{n\to \infty} \Xi_n&=\lim_{n\to \infty} \int_{\Rdd} \inner{\psi_n}{\pi(z)S^*S\pi(z)^*\phi}_{M^\infty_{1/v},M^1_v} \ dz \\
&= \int_{\Rdd} \lim_{n\to \infty} \inner{\psi_n}{\pi(z)S^*S\pi(z)^*\phi}_{M^\infty_{1/v},M^1_v} \ dz \\
&= \int_{\Rdd} \inner{\psi}{\pi(z)S^*S\pi(z)^*\phi}_{M^\infty_{1/v},M^1_v} \ dz.
\end{align*}
Hence we have shown that $$\int_{\Rdd} \inner{\psi}{\pi(z)S^*S\pi(z)^*\phi}_{M^\infty_{1/v},M^1_v} \ dz=\lim_{n\to \infty} \Xi_n = \|S\|_{\HS}^2 \inner{\psi}{\phi}_{M^\infty_{1/v},M^1_v},$$ which means that we are done once the interchange of the limit and integral has been justified. For each $n$ we may bound the integrand by 
{\small \begin{align*}
  |\inner{\psi_n}{\pi(z)S^*S\pi(z)^*\phi}_{M^\infty_{1/v},M^1_v}|&\leq \|\psi_n\|_{M^\infty_{1/v}}\cdot \|\pi(z)S^*S\pi(z)^*\phi\|_{M^1_v} \\
&\lesssim \|\psi\|_{M^\infty_{1/v}} \cdot v(z)\cdot   \|S^*\|_{\mathcal{N}} \cdot \|S\pi(z)^*\phi\|_{L^2} \quad \text{ by  \eqref{eq:tfbound}, \eqref{eq:opnormvsnuclear}}  \\
&= \|\psi\|_{M^\infty_{1/v}} \cdot \|S^*\|_{\mathcal{N}} \cdot v(z)\cdot \|\opstft_S(\phi)(z)\|_{L^2}.
\end{align*}}

Since $\phi \in M^1_{v}(\Rd)$, it follows by Lemma \ref{lem:analysisbound} that $z\mapsto v(z)\cdot \|\opstft_S(\phi)(z)\|_{L^2}$ is an integrable function. Hence we may apply the dominated convergence theorem.
\end{proof}
The proof of Theorem \ref{thm:continuous} is now straightforward. 
\begin{proof}[Proof of Theorem \ref{thm:continuous}]
The upper bound $\|\opstft_S(\psi)\|_{L^{p,q}_m(\Rdd;L^2)} \leq C_v^m \|S^*\|_{\mathcal{N}} \cdot \|\psi\|_{M^{p,q}_m}$ is the content of Lemma \ref{lem:analysisbound}.
By using the inversion formula and Lemma \ref{lem:synthesisbound} we obtain 
\begin{equation*}
  \|\psi\|_{M^{p,q}_m} = \frac{1}{\|S\|_{\HS}^2} \|\opstft_S^* \opstft_S\psi\|_{M^{p,q}_m}\leq \frac{C_v^m \|S^*\|_{\mathcal{N}}\cdot \|\varphi_0\|_{M^1_v}}{\|S\|_{\HS}^2}  \|\opstft_S(\psi)\|_{L^{p,q}_m(\Rdd;L^2)},
\end{equation*}
which implies the lower bound.
\end{proof}

\begin{rem}
	A different proof of a lower bound, more in line with the arguments in the proof of \cite[Prop. 2.2]{Grochenig:2011toft} (see Section \ref{sec:bonychemin} for more on this result), is to use that $S$ has a singular value decomposition 
	\begin{equation*}
S=\sum_{n=1}^\infty \lambda_n \eta_n \otimes \xi_n,
\end{equation*}
where $\lambda_n$ is a summable sequence of non-negative numbers and $\{\eta\}_{n=1}^\infty$, $\{\xi_n\}_{n=1}^\infty$ are orthonormal sequences in $L^2(\Rd)$. It is easy to check that since $S^*$ is bounded from $L^2(\Rd)$ to $M^1_v(\Rd)$, we must have $\xi_n \in M^1_v(\Rd)$ for all $n\in \N$. Then we find that
\begin{align*}
\|S\pi(z)^*\psi\|^2_{L^2}&=\left\| \sum_{n=1}^\infty \lambda_n V_{\xi_n} \psi (z) \phi_n \right\|_{L^2}^2 \\
&=\sum_{n=1}^\infty \lambda_n^2 |V_{\xi_n} \psi (z)|^2.
\end{align*}
Hence $\|S\pi(z)^*\psi\|_{L^2}\geq \lambda_1 |V_{\xi_1}\psi(z)|$, which leads to a lower bound by Theorem \ref{thm:windowindependent}. We have chosen to prove the lower bound in terms of $\opstft_S^*$ to emphasize the interpretation of our results as an STFT with operators as windows.
\end{rem}

As a first example we make sure that our result includes the well-known window independence from Theorem \ref{thm:windowindependent} as a special case. 
\begin{exmp} \label{exmp:recover}
As in Example \ref{exmp:finiterank}, we consider $\{\phi_n\}_{n=1}^N\subset M^1_v(\Rd),$ let $\{\xi_n\}_{n=1}^N$ be an orthonormal set in $L^2(\Rd)$ and define $S=\sum_{n=1}^N \xi_n \otimes \phi_n$. For $\psi \in M^{\infty}_{1/v}(\Rd)$ we then have $$\opstft_S(\psi)(z)=\sum_{n=1}^N V_{\phi_n}\psi(z)\xi_n.$$ By the orthonormality of the $\xi_n$'s we therefore have $$\|\opstft_S(\psi)(z)\|_{L^2}^2=\sum_{n=1}^N |V_{\phi_n}\psi(z)|^2.$$ It follows by Theorem \ref{thm:continuous} that $$C_{lower}\cdot\|\psi\|_{M^{p,q}_m}\leq \left\| \sqrt{\sum_{n=1}^N |V_{\phi_n}\psi|^2} \right\|_{L^{p,q}_m}\leq C_{upper}\cdot \|\psi\|_{M^{p,q}_m}.$$ In particular, if $N=1$ we recover Theorem \ref{thm:windowindependent} in the form $$C_{lower}\cdot \|\psi\|_{M^{p,q}_m}\leq \left\| V_{\phi_1}\psi \right\|_{L^{p,q}_m}\leq C_{upper} \cdot \|\psi\|_{M^{p,q}_m},$$ and it is easy to show that in this case 
\begin{align*}
C_{upper}&= C_{v}^m \cdot \|\phi_1\|_{M^1_v}\\
C_{lower}&= \|\phi_1\|_{L^2}^2\cdot (C_v^m \cdot \|\phi_1\|_{M^1_v}\cdot \|\varphi_0\|_{M^1_v})^{-1}.
\end{align*}
\end{exmp}

\section{The Weyl calculus and Bony-Chemin spaces} \label{sec:bonychemin}

In Section \ref{sec:hilbertschmidt} we defined Hilbert-Schmidt operators as integral operators, but any Hilbert-Schmidt operator can also be described as a \textit{Weyl operator.}  To define Weyl operators, we first introduce the \textit{cross-Wigner distribution} of  $\phi,\psi \in L^2(\Rd)$, which is the function
\begin{equation} \label{eq:wigner}
  W(\psi,\phi)(x,\omega)=\int_{\Rd} \psi(x+t/2)\overline{\phi(x-t/2)} e^{-2\pi i \omega\cdot t} \ dt \quad \text{ for } x,\omega\in \Rd.
\end{equation}
When $\psi = \phi$ we write $W(\psi)=W(\psi,\psi)$.
Given $\weyl \in L^2(\Rdd)$, we can define the Weyl operator $L_\weyl \in \HS$ by requiring that 
\begin{equation*}
\inner{L_\weyl \phi}{\psi}_{L^2}=\inner{\weyl}{W(\psi,\phi)}_{L^2} \quad \text{ for all } \psi,\phi \in L^2(\Rd).
\end{equation*}
The operator $L_\weyl$ is called the \textit{Weyl transform} of $\weyl$, and $\weyl$ is the \textit{Weyl symbol} of $L_\weyl.$ It is well-known that the Weyl transform $a\mapsto L_a$ is unitary from $L^2(\Rdd)$ to $\HS$. In particular, every $T\in \HS$ has a unique Weyl symbol $\weyl\in L^2(\Rdd)$ such that $T=L_\weyl.$

An interesting property of the Weyl symbol is its interaction with the time-frequency shifts. In fact, we have by \cite[Lem. 3.2]{Luef:2018c} that 
\begin{equation*}
\pi(z)L_\weyl \pi(z)^* = L_{T_z (\weyl)}.
\end{equation*}

Since $\pi(z)$ is unitary on $L^2(\Rd)$, this means that for $\weyl \in L^2(\Rdd)$ we have
\begin{equation*}
\|\opstft_{L_\weyl}\psi(z) \|_{L^2}=\|\pi(z) L_\weyl \pi(z)^* \psi\|_{L^2}=\|L_{T_z (\weyl)} \psi\|_{L^2}.
\end{equation*}

We may therefore reformulate Theorem \ref{thm:continuous} in terms of the Weyl transform.

\begin{thm} \label{thm:bonychemin}
	Let $0\neq \weyl \in L^2(\Rdd)$ such that $(L_a)^* \in \nuc$. For any $1\leq p,q \leq \infty$ and $v$-moderate weight $m$, we have
	\begin{equation*}
\|\psi\|_{M^{p,q}_m} \asymp  \left( \int_{\Rd} \left(\int_{\Rd} \|L_{T_{(x,\omega)}(a)}\psi\|_{L^2}^p m(x,\omega)^p \ dx \right)^{q/p} d\omega\right)^{\frac{1}{q}} .
\end{equation*}
\end{thm}

The above theorem generalizes a result by Gr\"ochenig and Toft in \cite[Prop. 2.2]{Grochenig:2011toft}, who showed that the the middle expression above defines an equivalent norm on $M^2_m(\Rd)$, i.e.
\begin{equation} \label{eq:grto}
\|\psi\|_{M^{2}_m}^2 \asymp   \int_{\Rd} \int_{\Rd} \|L_{T_{(x,\omega)}(a)}\psi\|_{L^2}^2 m(x,\omega)^2 \ dx  d\omega ,
\end{equation}
 under the assumptions that $m$ is of polynomial growth and $\weyl$ is a Schwartz function (stronger conditions are stated in \cite{Grochenig:2011toft}, but their proof uses only that $\weyl\in \mathscr{S}(\Rd)$). In fact, it is shown in \cite{Grochenig:2011toft} that the space of  $\psi \in \mathscr{S}'(\Rd)$ such that the right hand side of \eqref{eq:grto} is finite coincides with a space $H(m,g)$ introduced by Bony and Chemin \cite[Def. 5.1]{Bony:1994} when $g$ is the standard Euclidean metric on $\Rdd$. Hence \eqref{eq:grto} states that $H(m,g)=M^2_m(\Rd)$ with equivalent norms. 

Theorem \ref{thm:bonychemin} extends \eqref{eq:grto} in several directions. It extends from $p=q=2$ to any $1\leq p,q \leq \infty$ and from polynomial weights to general $v$-moderate weights. Our requirements on the Weyl symbol $\weyl$ are also weaker, although this is slightly obscured by the mysterious requirement that $(L_\weyl)^* \in \nuc$. By Proposition \ref{prop:projectivetensor} the condition $S^* \in \nuc$ means that the integral kernel $\kernel_S$ belongs to the projective tensor product $L^2(\Rd)\hat{\otimes} M^1_v(\Rd)$, and the Weyl symbol $\weyl$ and $\kernel_S$ are related by \cite{Heil:2003}
\begin{equation} \label{eq:kerweyl}
  k_S(x,y)=\int_{\Rd}\weyl\left(\frac{x+y}{2},\omega\right) e^{2\pi i \omega \cdot (x-y)} \ d\omega.
\end{equation}
Understanding the condition $(L_\weyl)^*\in \nuc$ thus boils down to understanding what assumptions we need on $\weyl$ to ensure that the kernel $k_S$ in \eqref{eq:kerweyl} belongs to $L^2(\Rd)\hat{\otimes} M^1_v(\Rd)$.

\subsection{Polynomial weights}
By restricting our attention to the polynomial weights $v_s(z)=(1+|z|^2)^{s/2}$ for $s\geq 0$, we obtain some sufficient conditions for $(L_\weyl )^*\in \mathcal{N}(L^2,M^1_{v_s})$, so that Theorem \ref{thm:bonychemin} holds. 

\begin{exmp}[Schwartz symbols]
	If $v=v_s$ for $s\geq 0$, we know from Example \ref{exmp:schwartz} that the Schwartz operators $\mathfrak{S}$, i.e. operators $T$ with $\kernel_T \in \mathscr{S}(\Rdd)$, form a subspace of $\nuc$. Furthermore, the space $\mathfrak{S}$ is closed under taking adjoints, and may equivalently be described as the Weyl operators $L_\weyl$ with $\weyl \in \mathscr{S}(\Rdd)$ \cite{Keyl:2015}. Taken together, this means that $\weyl \in \mathfrak{S}$ implies $(L_\weyl)^*\in \mathfrak{S} \subset \mathcal{N}(L^2,M^1_{v_s})$. Thus Theorem \ref{thm:bonychemin} applies for all Schwartz functions $\weyl$.
	\end{exmp}
	
	We then prove a slightly more refined result. Below we denote by $v_s^{4d}$ the weight function on $\R^{4d}$ given by $v_s^{4d}(z,\zeta)=(1+|z|^2+|\zeta|^2)^{s/2}$. 	
	\begin{prop}
		If $\weyl \in M^1_{v_{2s}^{4d}}(\Rdd)$ for $s\geq 0$, then $(L_\weyl)^* \in \mathcal{N}(L^2;M^1_{v_{s}})$. Hence Theorem \ref{thm:bonychemin} applies with $v=v_{s}$. 
	\end{prop}
	\begin{proof}
		 Recall from \eqref{eq:projtensorfeichtinger}  that with $v_s\tilde{\otimes} v_s(x_1,x_2,\omega_1,\omega_2)=v_s(x_1,\omega_1)\cdot v_{s}(x_2,\omega_2),$ we have the equality $M^1_{v_{s}\tilde{\otimes} v_{s}}(\Rdd)=M^1_{v_{s}}(\Rd)\hat{\otimes}M^1_{v_{s}}(\Rd)$. One easily checks that $v_{s}\tilde{\otimes} v_{s} \lesssim v^{4d}_{2s}$, which implies by part b) of Proposition \ref{prop:modulationspaces} and Proposition \ref{prop:projectivetensor} that $$M^1_{v_{2s}^{4d}}(\Rdd)\hookrightarrow M^1_{v_s \tilde{\otimes}v_s}(\Rdd)=M^1_{v_{s}}(\Rd)\hat{\otimes}M^1_{v_{s}}(\Rd)\hookrightarrow L^2(\Rd)\hat{\otimes} M^1_{v_s}(\Rd).$$
		
		By \cite[Prop. 7.4.1]{Heil:2003}, if $\weyl \in M^1_{v_{2s}^{4d}}(\Rdd)$ then the integral kernel $\kernel_{L_a}$ also satisfies $\kernel_{L_\weyl}\in M^1_{v_{2s}^{4d}}(\Rdd)$. By the chain of inclusions above, it follows that $\kernel_{L_\weyl}\in L^2(\Rd)\hat{\otimes} M^1_{v_s}(\Rd)$, which implies $(L_\weyl)^*\in \mathcal{N}(L^2,M^1_{v_s})$ by Proposition \ref{prop:projectivetensor}.
		\end{proof}
	
	When $s=0$ the condition above is rather weak, as $M^1(\Rdd)$ even contains non-differentiable functions.

\section{Cohen's class} \label{sec:cohensclass}

Another interesting interpretation of Theorem \ref{thm:continuous} is in terms of Cohen's class of time-frequency distributions introduced by Cohen in \cite{Cohen:1966}. Typically the definition of the Cohen's class distribution $Q_a$ associated with $a\in \mathscr{S}'(\Rdd)$ is that \cite{Grochenig:2001}
\begin{equation} \label{eq:cohentrad}
Q_a(\psi) = a \ast W(\psi) \quad \text{ for any } \psi \in \mathscr{S}(\Rd).
\end{equation}   
One can show that $\psi \in \mathscr{S}(\Rd)$ implies that $W(\psi)\in \mathscr{S}(\Rdd)$, so \eqref{eq:cohentrad} is well-defined as the convolution of a tempered distribution with a Schwartz function. All our examples will satisfy $a\in L^2(\Rdd)$, and in this case $Q_a(\psi)$ is defined by \eqref{eq:cohentrad} for any $\psi \in L^2(\Rd)$, as a slight modification of Moyal's identity gives that $W(\psi)\in L^2(\Rdd)$, so \eqref{eq:cohentrad} is well-defined by Young's inequality. 

In \cite{Luef:2018b} we have given an alternative description of Cohen's class. Given a Hilbert-Schmidt operator $T\in \HS$, we define the Cohen's class distribution $Q_T$ associated with $T$ by
\begin{equation} \label{eq:cohenclassgeneraloperator}
  Q_T(\psi)(z)= \inner{T\pi(z)^*\psi}{\pi(z)^* \psi}_{L^2}.
\end{equation}

Any Cohen class distribution $Q_\weyl$ for $\weyl \in L^2(\Rdd)$ can equivalently be described using \eqref{eq:cohenclassgeneraloperator}, since it follows from \cite[Prop. 7.1]{Luef:2018b} that $$Q_\weyl(\psi)=Q_{L_{\check{\weyl}}}(\psi),$$ where $L$ denotes the Weyl transform and $\check{\weyl}(z)=\weyl(-z)$.  From now on we will therefore  write Cohen's class distributions in the form $Q_T$ for $T\in \HS$ rather than using \eqref{eq:cohentrad}. 

In light of \eqref{eq:cohenclassgeneraloperator} we clearly have the relation 
\begin{equation} \label{eq:stftcohenclass}
\|\opstft_{S}(\psi)(z)\|^2_{L^2}= \inner{S\pi(z)^*\psi}{S\pi(z)^*\psi}_{L^2} = Q_{S^*S}(\psi)(z).
\end{equation}

Hence $\|\opstft_{S}(\psi)(z)\|^2_{L^2}=\sqrt{Q_{S^*S}(\psi)(z)}$, and we see that another reinterpretation of Theorem \ref{thm:continuous} is the following.

\begin{thm} \label{thm:cohenclass}
	Let $0\neq S\in \HS$ such that $S^*\in \nuc$. For any $1\leq p,q \leq \infty$ and $v$-moderate weight $m$, we have
	\begin{equation*} 
  \|\psi\|_{M^{p,q}_m} \asymp \left\|\sqrt{Q_{S^*S}(\psi)}\right\|_{L^{p,q}_{m}(\Rdd)} .
\end{equation*}
\end{thm}

\begin{exmp}[Spectrograms]
To see why the square root appears in Theorem \ref{thm:cohenclass}, it is worth recalling the simple case of $S=\xi \otimes \phi$ for some $\phi \in M^1_v(\Rd)$ and $\|\xi\|_{L^2}=1$. Then $S^*S=\phi \otimes \phi$, and one may check that 
\begin{equation*} 
Q_{S^*S}(\psi)(z)=|V_\phi \psi(z)|^2.
\end{equation*}
This is the so-called \textit{spectrogram} of $\psi$ with window $\phi$, and we know from Theorem \ref{thm:windowindependent} that $\|\psi\|_{M^{p,q}_m}\asymp \|V_\phi \psi\|_{L^{p,q}_m}$, hence we need the square root in Theorem \ref{thm:cohenclass}.  
\end{exmp}
 
\begin{rem} 
We have skipped one technical detail in the Theorem \ref{thm:cohenclass} above, namely how to interpret $Q_{T}(\psi)$ for $\psi \in M^\infty_{1/v}(\Rdd)$. This is certainly not immediately covered by \eqref{eq:cohentrad} or \eqref{eq:cohenclassgeneraloperator}. We solve this issue by rewriting  $Q_T(\psi)$ to
	\begin{equation*}
Q_T(\psi)=\inner{\pi(z)^*\psi}{T^*\pi(z)^*\psi}_{L^2}
\end{equation*} 
and then replacing the bracket by duality:
\begin{equation} \label{eq:cohenrem2}
Q_T(\psi):=\inner{\pi(z)^*\psi}{T^*\pi(z)^*\psi}_{M^\infty_{1/v},M^1_v}.
\end{equation}
This defines $Q_T(\psi)$ for $\psi \in M^\infty_{1/v}(\Rdd)$ whenever $T^*$ maps $M^\infty_{1/v}(\Rd)$ into  $M^1_v(\Rd)$, which is true if $T=S^*S$ for $S^*\in \nuc$ or if $k_T \in M^1_{v\otimes v}(\Rdd)$, see \cite[Prop. 4.1]{Skrettingland:2019} for a proof. It is straightforward to check that \eqref{eq:cohenclassgeneraloperator} and \eqref{eq:cohenrem2} agree when $\psi \in L^2(\Rd)$, and that $Q_{T}(\psi)(z)=\|\opstft_{S}(\psi)(z)\|^2_{L^2} $ when $T=S^*S$.

\end{rem}
\subsection{On positive Cohen class distributions} \label{sec:poscohen}
The reader will not fail to notice that the Cohen class distributions for which Theorem \ref{thm:cohenclass} applies are of a particular kind, namely of the form $Q_T$ for $T=S^*S$ with $S^*\in \nuc$.

 The condition $S^*\in \nuc$ may be interpreted as requiring a certain time-frequency localization for $Q_{S^*S}$, as one can show that $S^*\in \nuc$ implies that the integral kernel $k_{S^*S}$ belongs to $M^1_v(\Rd)\hat{\otimes} M^1_v(\Rd)$. If $S=\xi \otimes \phi$ for $\|\xi\|_{L^2}=1$, which we know from Example \ref{exmp:recover} corresponds to choosing the window $\phi$ for the modulation spaces, then $S^*S=\phi \otimes \phi$, which has integral kernel in $M^1_v(\Rd)\hat{\otimes}M^1_v(\Rd)$ precisely when $\phi \in M^1_v(\Rd)$. Hence requiring $S^* \in \nuc$ seems like a natural generalization of the assumption in Theorem \ref{thm:windowindependent} that windows $\phi$ for modulation spaces need to satisfy $\phi\in M^1_v(\Rd)$.
 
In addition, the fact that $T=S^*S$ means that $T$ is a positive operator. By \cite[Prop. 7.3]{Luef:2018b}, this is equivalent to $Q_T(\psi)$ being a positive function for each $\psi \in L^2(\Rd).$ This assumption cannot simply be replaced by considering $|Q_T(\psi)|$, as the following example shows. 

\begin{exmp}
	Let $\phi_1$ and $\phi_2$ be compactly supported functions in $\mathscr{S}(\Rd)$ such that their supports do not overlap. Define $T=\phi_1\otimes \phi_2$. Then the integral kernel (or equivalently the Weyl symbol) of $T$ belongs to $\mathscr{S}(\Rdd)$, and has good time-frequency localization in this sense. However, $T$ is not a positive operator as $\phi_1 \neq \phi_2$, and Theorem \ref{thm:cohenclass} fails when replacing $Q_{S^*S}$ by $|Q_T|$: for instance, one easily finds using  that \eqref{eq:cohenrem2} that when $\delta$ is the Dirac distribution
	\begin{equation*}
Q_T(\delta)(z)=\phi_1(x)\overline{\phi_2(x)}\equiv 0.
\end{equation*}
\end{exmp}

An obvious question is whether the positivity and good time-frequency properties exhibited by $Q_{S^*S}$ when $S^*\in \nuc$ are sufficient for Theorem \ref{thm:cohenclass} to hold: 
\begin{quote}
	If $T\in \HS$ has integral kernel in $M^1_v(\Rd)\hat{\otimes} M^1_v(\Rd)$ and is a positive operator on $L^2(\Rd)$, does a version of Theorem \ref{thm:cohenclass} hold with $Q_{S^*S}$ replaced by $Q_T$?
\end{quote}

As a first step in this direction, we note that the statement is true if $T\in \mathfrak{S}$, i.e. if $k_T\in \mathscr{S}(\Rdd)$, as \cite[Prop. 3.15]{Keyl:2015} states that if $T\in \mathfrak{S}$ is positive, then $\sqrt{T}\in \mathfrak{S}$.

\begin{thm} \label{thm:squareroot}
	Let $T\in \mathfrak{S}$ be a positive operator, and assume that $v$ grows at most polynomially. Then, for any $1\leq p,q \leq \infty$ and $v$-moderate weight $m$, we have
	\begin{equation*}
  \|\psi\|_{M^{p,q}_m} \asymp \left\|\sqrt{Q_{T}(\psi)}\right\|_{L^{p,q}_{m}} .
\end{equation*}
\end{thm}

\begin{proof}
	As noted, $S:=\sqrt{T}\in \mathfrak{S}$. Then $T=S^*S$, and we saw in Example \ref{exmp:schwartz} that $S\in \mathfrak{S}$ implies that $S\in \nuc$ under the assumption that $v$ grows at most polynomially. The result therefore follows by Theorem \ref{thm:cohenclass}.
\end{proof}

This theorem can also be formulated using the classic definition \eqref{eq:cohentrad} of Cohen's class. In this formulation it states that if $a\in \mathscr{S}(\Rdd)$ and $Q_a(\psi)$ is a positive function for each $\psi \in L^2(\Rd),$ then $\|\psi\|_{M^{p,q}_m} \asymp \|\sqrt{Q_{a}(\psi)}\|_{L^{p,q}_{m}}$. 

A question for further research is then if the same holds for $M^1_{v\otimes v}(\Rdd)$: if $T$ is a positive operator with $k_T\in M^1_{v\otimes v}(\Rdd)$, what can we say about $k_{\sqrt{T}}$? 

\section{Examples: Localization operators} \label{sec:locops}
We now return to the localization operators considered in Section \ref{sec:locop} by choosing $S=\mathcal{A}^{\varphi_1,\varphi_2}_f$ with $\varphi_1\in M^1_v(\Rd)$, $\varphi_2 \in L^2(\Rd)$ and $f\in L^1_v(\Rdd)$. Then $S^*\in \nuc$ by Corollary \ref{cor:locopbound}. To apply Theorem \ref{thm:continuous} to this example, we first note that a calculation gives
\begin{equation*}
\pi(z) \mathcal{A}^{\varphi_1,\varphi_2}_f \pi(z)^* = \mathcal{A}^{\varphi_1,\varphi_2}_{T_z(f)},
\end{equation*} 
i.e. conjugating the localization operator by $\pi(z)$ amounts to translating $f$ by $z$. As we saw in Section \ref{sec:bonychemin} we also have by the unitarity of $\pi(z)$ that
\begin{equation} \label{eq:shiftingstft}
\|\opstft_{\mathcal{A}^{\varphi_1,\varphi_2}_f } \psi(z) \|_{L^2}=\|\pi(z)\mathcal{A}^{\varphi_1,\varphi_2}_f \pi(z)^*\psi\|_{L^2}=\|\mathcal{A}^{\varphi_1,\varphi_2}_{T_z(f)}(\psi)\|_{L^2},
\end{equation} 
 hence we obtain the following from Theorem \ref{thm:continuous}.

\begin{thm} \label{thm:locop}
	Assume that $\varphi_1\in M^1_v(\Rd)$, $\varphi_2 \in L^2(\Rd)$ and $f\in L^1_v(\Rdd)$. For any $1\leq p,q \leq \infty$ and $v$-moderate weight $m$, we have
	\begin{equation*}
  \left( \int_{\Rd} \left( \int_{\Rd} \left\|\mathcal{A}^{\varphi_1,\varphi_2}_{T_{(x,\omega)}(f)}(\psi)\right\|_{L^2}^p m(x,\omega)^p\ dx\right)^{q/p} \ d\omega \right)^{1/q}\  \asymp   \|\psi\|_{M^{p,q}_m},
\end{equation*}
 where the integrals are replaced by supremums if $p=\infty$ or $q=\infty$.
\end{thm}
In light of \eqref{eq:stftcohenclass} and \eqref{eq:shiftingstft}, we know that 
\begin{equation*}
 \left\|\mathcal{A}^{\varphi_1,\varphi_2}_{T_{(x,\omega)}(f)}(\psi)\right\|_{L^2}^2=Q_{T}(\psi)
\end{equation*} 
where $T=\left(\mathcal{A}^{\varphi_1,\varphi_2}_{f}\right)^*\mathcal{A}^{\varphi_1,\varphi_2}_{f}=\mathcal{A}^{\varphi_2,\varphi_1}_{\overline{f}}\mathcal{A}^{\varphi_1,\varphi_2}_{f}$. In this sense Theorem \ref{thm:locop} concerns the study of a particular kind of Cohen's class distribution. 

\begin{rem}
	We mention that there is another line of research that leads to equivalent norms for modulation spaces in terms of localization operators, see \cite{Boggiatto:2005,Grochenig:2011toft,Grochenig:2013}. In this approach one considers weights function $m,m_0$ and shows that under various conditions on $m$ the localization operator $\mathcal{A}^{\varphi,\varphi}_m$ is an isomorphism from $M^{p,q}_{m_0}(\Rd)$ to $M^{p,q}_{m_0/m}(\Rd).$ This implies a norm equivalence $\|\psi\|_{M^{p,q}_{m_0}}\asymp \|\mathcal{A}^{\varphi,\varphi}_{m}\psi\|_{M^{p,q}_{m_0/m}}$, which is of a different nature than the one we consider. 
\end{rem}

\subsection{Modulation spaces as time-frequency Wiener amalgam spaces}
A consequence of Theorem	 \ref{thm:locop} is that we may interpret modulation spaces as a time-frequency version of the so-called \textit{Wiener amalgam spaces} \cite{Feichtinger:1983}; a class of function function spaces that have been closely tied to the development of modulation spaces since the inception of the latter in \cite{Feichtinger:1983}. To explain this interpretation, we start by considering for $\varphi\in M^1_v(\Rd)$ and $f\in L^1_v(\Rdd)$ the localization operator
\begin{equation} \label{eq:locopinterpret}
\mathcal{A}_{f}^{\varphi,\varphi}(\psi)=\int_{\Rdd} f(z) V_{\varphi}\psi(z) \pi(z) \varphi \ dz.
\end{equation}

In time-frequency analysis, when $\varphi$ is well-localized in time and frequency such as the Gaussian, the size of $|V_\varphi \psi (x,\omega)|$ is interpreted as a measure of the contribution of the frequency $\omega$ at time $x$ of the signal $\psi$. By the reconstruction formula 
\begin{equation} \label{eq:reconstruct}
\psi = \frac{1}{\|\varphi\|_{L^2}^2}\int_{\Rdd} V_\varphi \psi(z) \pi(z) \varphi \ dz
\end{equation}
we can recover $\psi$ from $V_\varphi \psi$, and \eqref{eq:locopinterpret} finds a natural interpretation as a multiplication operator in the time-frequency plane: we represent $\psi$ in the time-frequency plane by forming $V_\varphi \psi$, but before we reconstruct $\psi$ from $V_\varphi \psi$ we multiply it by $f(z)$. A particular choice of $f$ is to let $f$ be the characteristic function $\chi_\Omega$ for some compact subset $\Omega$. Then
\begin{equation*}
\mathcal{A}^{\varphi,\varphi}_{T_z(\chi_\Omega)}(\psi)=\int_{\Rdd} \chi_{z+\Omega}(z') V_\varphi \psi(z') \pi(z') \psi \ d z',
\end{equation*}
which in light of \eqref{eq:reconstruct} may be interpreted as saying that $\mathcal{A}^{\varphi,\varphi}_{T_z(\chi_\Omega)}$ picks out the component of $\psi$ localized in $z+\Omega:= \{z+z':z'\in \Omega\}$ in the time-frequency plane. Theorem \ref{thm:locop} says that an equivalent norm on $M^{p,q}_m(\Rd)$ is given by first measuring the \textit{local} size of $\psi$ near $z$ in the time-frequency plane by $\|\mathcal{A}^{\varphi,\varphi}_{T_z(\chi_\Omega)}\psi\|_{L^2}$, and then measuring the \textit{global} properties of $\psi$ by taking the $L^{p,q}_m$ norm.

When $p=q$, this parallels the definition of the Wiener amalgam space $W(L^2,L^p_w)$ with local component $L^2$ and global component $L^{p}_w$.
For a fixed, compact domain $Q\subset \Rd$, the Wiener amalgam space $W(L^2,L^{p}_w)$ for $1\leq p\leq \infty$ and a weight function $w$ on $\Rd$ consists of all functions $\psi:\Rd\to \mathbb{C}$ such that 
\begin{equation*} 
 \|\psi\|_{W(L^2,L^{p,q}_m)}:= \left( \int_{\Rd} \left\|\chi_{x+Q}\cdot \psi\right\|_{L^2}^p w(x)^p \ dx\right)^{1/p} .
\end{equation*}
Since we interpret $\mathcal{A}_{T_z(\chi_\Omega)}^{\varphi,\varphi}(\psi)$ as $\psi$ localized to $z+\Omega$ in the time-frequency plane and $\chi_{x+Q}\cdot \psi$ is the localization of $\psi$ to $x+Q$ in time, Theorem \ref{thm:locop} says that modulation spaces are the natural analogues of Wiener amalgam spaces when  we localize $\psi$ in the time-frequency plane using $\mathcal{A}^{\varphi,\varphi}_{\chi_\Omega}$, not just in time by multiplying with $\chi_Q$.
We have merely scratched the surface of Wiener amalgam spaces, and the interested reader should consult the survey \cite{Heil:2003am}. However, it is worth noting that when the cutoff-function $\chi_Q$ is replaced by a smooth cutoff-function $\phi$ satisfying $\sum_{\ell\in \mathbb{Z}^{d}} T_{\ell}(\phi)\equiv 1$, then an equivalent norm on $W(L^2,L^p_m)$ is given by
\begin{equation*}
\left(\sum_{\ell\in \mathbb{Z}^d} \|T_\ell(\phi) \cdot \psi\|_{L^2}^p w(\ell)^p\right)^{1/p}.
\end{equation*}
The fact that modulation spaces have a similar discrete description has already been shown by D\"orfler, Feichtinger and Gr\"ochenig in \cite{Dorfler:2006,Dorfler:2011}: if $f\in L^1_v(\Rdd)$ satisfies $$\sum_{(j,k)\in \Z^{2d}} T_{(j,k)}(f)\asymp 1,$$ then an equivalent norm on $M^{p,q}_m(\Rd)$ is given by
\begin{equation*}
 \left( \sum_{k\in \Z^d} \left( \sum_{j\in \Z^d} \|\mathcal{A}^{\varphi,\varphi}_{T_{(j,k)}(f)}(\psi)\|_{L^2}^p m(j,k)^p \right)^{p/q} \right)^{1/q}.
\end{equation*}
Finally, we remark that the local component $L^2$ in $W(L^2,L^p_w)$ can be replaced by several other function spaces $X$ to obtain new Wiener amalgam spaces $W(X,L^p_m).$ One might therefore naturally replace the $L^2$-norm in Theorem \ref{thm:locop} or Theorem \ref{thm:continuous} by another function space norm and investigate the resulting function spaces. 

\subsection{Smoothing spectrograms}

So far in this section we have picked $S$ to be a localization operator $\mathcal{A}_f^{\varphi_1,\varphi_2}$, which corresponds to studying the Cohen's class distribution $Q_T$ for $T=\mathcal{A}^{\varphi_2,\varphi_1}_{\overline{f}}\mathcal{A}^{\varphi_1,\varphi_2}_{f}$. However, we may also proceed as in Section \ref{sec:poscohen} and study the Cohen class distribution $Q_T$ for $T=\mathcal{A}^{\varphi,\varphi}_f$, where $f\in L^1_v(\Rdd)$ is a non-negative function and $\psi \in M^1_v(\Rd)$. The fact that $f$ is non-negative implies that $T$ is a positive operator, and it is not difficult to show that 
\begin{equation*}
Q_{\mathcal{A}^{\varphi,\varphi}_f}(\psi)(z)=f\ast |V_\varphi \psi(z)|^2(z),
\end{equation*}
i.e. the Cohen class of $\mathcal{A}^{\varphi,\varphi}_f$ is a smoothed spectrogram. Theorem \ref{thm:cohenclass} says that if $\mathcal{A}^{\varphi,\varphi}_f=S^*S$ for some $S^*\in \nuc$, then 
$$\|\psi\|_{M^{p,q}_m} \asymp \left\|\sqrt{f\ast |V_\varphi \psi|^2}\right\|_{L^{p,q}_{m}(\Rdd)}.$$

As we discussed in Section \ref{sec:poscohen}, the existence of such $S$ is not clear in general, but if $\mathcal{A}^{\varphi,\varphi}_f \in \mathfrak{S}$ we can use Theorem \ref{thm:squareroot} to deduce the following result.

\begin{prop}
	Let $\varphi\in \mathscr{S}(\Rd)$ and let $f\in L^1(\Rdd)$ be a positive function of compact support. If $v$ grows at most polynomially and $m$ is $v$-moderate, then  
	\begin{equation*}
\|\psi\|_{M^{p,q}_m} \asymp \left\|\sqrt{f\ast |V_\varphi \psi|^2}\right\|_{L^{p,q}_{m}(\Rdd)}.
\end{equation*}
\end{prop} 
\begin{proof}
	The Weyl symbol of $\mathcal{A}^{\varphi,\varphi}_{f}$ is the function $f\ast W(\varphi)$, see for instance \cite[Lem. 2.4]{Boggiatto:2004}. As $\varphi \in \mathscr{S}(\Rd)$ it follows by \cite[Lem. 14.5.1]{Grochenig:2001} that $W(\phi)\in \mathscr{S}(\Rdd)$. Hence the assumptions on $f$ imply that $f\ast W(\varphi)\in \mathscr{S}(\Rdd)$, which means that $\mathcal{A}_f^{\varphi,\varphi}\in \mathfrak{S}.$ The result therefore follows by Theorem \ref{thm:squareroot}.
\end{proof}

\appendix
\section{Proof of Proposition \ref{prop:underspread}}

\begin{proof}[Proof of Proposition \ref{prop:underspread}] 
First recall from Section \ref{sec:bonychemin} that $\beauty$ consists precisely of those $T\in \HS$ such that the Weyl symbol $\weyl_T$ belongs to $M^1(\Rdd)$.
Then recall that we assume 
\begin{equation*}
T=\int_{\Rdd} F(x,\omega) e^{-i\pi x\cdot \omega} \pi(x,\omega) \ dxd\omega,
\end{equation*}
where $F(x,\omega)\in L^2(\Rdd)$ has compact support, say $\text{supp}(F)\subset K$. As in \cite{Luef:2018c}, we denote the function $F$ by $\F_W(T)$ -- it plays the role of a Fourier transform of the operator $T$ in quantum harmonic analysis. 

 One can show that 
\begin{equation*}
\F_W(T)=\F_\sigma(\weyl_T),
\end{equation*} 
where $\F_\sigma(f)$ is the \textit{symplectic Fourier transform} of $f\in L^1(\Rdd)$ given by
	\begin{equation*}
\F_{\sigma} f(x,\omega)=\int_{\R^{2d}} f(x',\omega') e^{-2 \pi i (x'\cdot \omega - x\cdot \omega')} \ dx'd\omega' \quad \text{ for } x,x',\omega,\omega' \in \Rd.
\end{equation*}

Then fix some $R\in \beauty$ such that $\F_W(R)$ has no zeros, an explicit example is $R=\varphi_0\otimes \varphi_0$ \cite[Ex. 6.1]{Luef:2018c}. As $R\in \beauty$, we have  $\weyl_R=\F_\sigma\F_W(R)\in M^1(\Rdd)$.

Since $\weyl_R\in L^1(\Rdd)$ and $\F_\sigma(\weyl_R)=\F_W(R)$ never vanishes, the Wiener-L\'{e}vy theorem \cite[Thm. 3.1]{Reiter:2000} implies the existence of some $h\in L^1(\Rdd)$ such that $$\F_\sigma(h)(z)=\frac{1}{\F_\sigma(\weyl_R)}=\frac{1}{\F_W(R)(z)}\quad \text{ for } z\in K.$$ 

Then define the operator $$T'=(h\ast \weyl_R) \star T,$$ where $\star$ is the operation from \eqref{eq:opconv}. The "Fourier transform" $\F_W$ interacts with the convolutions in the expected way \cite[Prop. 6.4]{Luef:2018c}; more precisely, we have that 
$$\F_W(T')=\F_\sigma(h)\F_W(R)\F_W(T)=\F_W(T)$$ by construction of $h$. As $\F_W$ is injective, see \cite[Cor. 7.6.3]{Feichtinger:1998}, it follows that $T=T'$.

On the other hand, the function $b:=h\ast \weyl_R$ belongs to $M^1(\Rdd)$ since $L^1(\Rdd)\ast M^1(\Rdd)\subset M^1(\Rdd)$ by \cite[Prop. 12.1.7]{Grochenig:2001}. The Weyl symbol of $T=T'=b\star T$ is given by $\weyl_T=b\ast \weyl_T$ \cite[Prop. 5.2]{Luef:2018b}. Since $b\in M^1(\Rdd)$ and $T\in \tco$, \cite[Thm. 8.1]{Luef:2018c} implies\footnote{The theorem states that if $S\in \beauty$ and $T\in \tco$, then $S\star T\in M^1(\Rdd)$. $S\star T$ is just another notation for $\weyl_{S}\ast \weyl_T$, so our result follows by picking $S=L_b$.} that $\weyl_T = b\ast \weyl_T\in M^1(\Rdd)$, hence $T\in \beauty$.  
\end{proof}

\section*{Acknowledgements}
The author acknowledges helpful feedback from Franz Luef on various drafts of the paper.

\end{document}